\documentclass[12pt,article]{amsart}
\usepackage{mathrsfs}
\usepackage{amssymb}
\usepackage{amsfonts}
\usepackage{amsbsy}
\usepackage{latexsym}
\usepackage{amssymb,latexsym,amsmath,amsthm}
\usepackage{framed}
\usepackage[colorlinks,linkcolor=blue]{hyperref}
\usepackage{graphicx}
\usepackage{xcolor}
\usepackage{epstopdf}
\usepackage{bm}
\usepackage[normalem]{ulem}
\theoremstyle{plain}

 \theoremstyle{remark} 

\newtheorem {theo} {\bf Theorem} [section]
\newtheorem {prop} [theo] {\bf Proposition}

\newtheorem {lem} [theo] {\bf Lemma}

\newtheorem {defi} {\bf Definition}[section]
\newtheorem{exam} {\bf Example}[section]

\newtheorem{rem}{\bf Remark}[section]

\numberwithin{equation}{section}
\begin{document}
\title[Single-angle Radon samples based  determination]{ Determination     of  compactly supported  functions in   shift-invariant space by single-angle Radon samples}
\author{Youfa Li}
\address{College of Mathematics and Information Science\\
Guangxi University,  Nanning, China }
\email{youfalee@hotmail.com}
\author{Shengli Fan}
\address{CREOL College of Optics \& Photonics\\
University of Central Florida,  Orlando, FL 32816}
\email{shengli.fan@knights.ucf.edu}
\author{Deguang Han}
\address{Department of Mathematics,
University of Central Florida,  Orlando, FL 32816}
\email{Deguang.Han@ucf.edu}
\thanks{
Youfa Li is partially supported by Natural Science Foundation of China (Nos: 61961003, 61561006, 11501132),  Natural Science Foundation of Guangxi (Nos: 2018JJA110110, 2016GXNSFAA380049) and  the talent project of  Education Department of Guangxi Government  for Young-Middle-Aged backbone teachers.
Deguang Han  is partially supported by the NSF grants  DMS-1712602 and DMS-2105038.
}

\date{\today}

\begin{abstract} While traditionally the computerized tomography  of a  function  $f\in L^{2}(\mathbb{R}^{2})$ depends on the samples of its  Radon transform at multiple angles, the real-time imaging sometimes requires the reconstruction   of $f$  by the   samples of its Radon transform $\mathcal{R}_{\emph{\textbf{p}}}f$ at a single angle $\theta$, where
$\emph{\textbf{p}}=(\cos\theta, \sin\theta)$ is the direction vector. This naturally leads to the question of identifying those functions that can be determined   by their  Radon  samples    at a single angle $\theta$.
The    shift-invariant space  $V(\varphi, \mathbb{Z}^2)$ generated by $\varphi$ is  a   type of  function space that has been  widely  considered  in  many fields  including  wavelet  analysis and  signal processing.   In this paper  we examine  the  single-angle reconstruction  problem  for compactly supported functions $f\in V(\varphi, \mathbb{Z}^2)$. The central  issue  for the  problem
is to identify  the eligible  $\emph{\textbf{p}}$ and sampling set $X_{\emph{\textbf{p}}}\subseteq \mathbb{R}$
such that $f$ can be determined   by its single-angle Radon (w.r.t $\emph{\textbf{p}}$) samples at $X_{\emph{\textbf{p}}}$.
For the general generator $\varphi$, we address the eligible  $\emph{\textbf{p}}$ for the two cases: (1) $\varphi$ being nonvanishing ($\int_{\mathbb{R}^{2}}\varphi(\emph{\textbf{x}})d\emph{\textbf{x}}\neq0$) and (2) being  vanishing ($\int_{\mathbb{R}^2}\varphi(\emph{\textbf{x}})d\emph{\textbf{x}}=0$). We prove that eligible  $X_{\emph{\textbf{p}}}$ exists for general $\varphi$. In particular, $X_{\emph{\textbf{p}}}$
can be explicitly constructed if  $\varphi\in C^{1}(\mathbb{R}^{2})$.
Positive definite functions  form an important  class of functions that have been widely applied
in scattered data interpolation. The single-angle problem corresponding to the case that
$\varphi$ being positive definite is addressed such that  $X_{\emph{\textbf{p}}}$
can be constructed  easily. Besides using the samples of the single-angle  Radon transform, another   common feature for our recovery  results is that the
number of the  required samples is minimum.
\end{abstract}
\maketitle
\section{Introduction}\label{diyijiesection}
\subsection{CT and Radon transform}\label{subsection1.1}
We start with the X-ray computerized tomography (CT) on $\mathbb{R}^{2}$. 
Its core mathematics   includes  the Radon transform and its inversion.
For a function $f: \mathbb{R}^2\longrightarrow \mathbb{C}$
its   Radon transform at $t\in \mathbb{R}$, w.r.t a   direction vector   $\emph{\textbf{p}}=(\cos\theta, \sin\theta)$,
is defined as the integral of $f$ along the line $(x,y)=
t\emph{\textbf{p}}+s(-\sin\theta, \cos\theta)$ on $ \mathbb{R}^2:$
\begin{align}\label{gbhf} \mathcal{R}_{\emph{\textbf{p}}}f(t):=\int^{\infty}_{-\infty}
f(t\cos\theta-s\sin\theta, t\sin\theta+s\cos\theta)ds. \end{align}
If $f\in L^{1}(\mathbb{R}^{2})$ then we  can  prove  that $\mathcal{R}_{\emph{\textbf{p}}}f\in L^{1}(\mathbb{R})$:
 \begin{align}\begin{array}{llll}\displaystyle \int^{\infty}_{-\infty}|\mathcal{R}_{\emph{\textbf{p}}}f(t)|dt&=
 \displaystyle\int^{\infty}_{-\infty} \Big|\int^{\infty}_{-\infty}f((t,s)A)ds\Big|dt\\
 &\displaystyle\leq\int^{\infty}_{-\infty} \int^{\infty}_{-\infty}|f((t,s)A)|dsdt\\
 &=\|f\|_{L^{1}(\mathbb{R}^2)}
 \end{array}\end{align}
 where $A=\left(
\begin{array}{cccc}
   \cos\theta&\sin\theta \\
   -\sin\theta&\cos\theta
\end{array}
\right).$
If $f\in L^{2}(\mathbb{R}^{2})$
 is compactly supported,  then by the  Cauchy-Schwarz inequality   one can check  that $\mathcal{R}_{\emph{\textbf{p}}}f\in L^{2}(\mathbb{R})$.
The Fourier transforms of $\mathcal{R}_{\emph{\textbf{p}}}f$ and   $f$   are  correlated   via
 \begin{align}\label{Radonform} \widehat{\mathcal{R}_{\emph{\textbf{p}}}f}(\xi)=\widehat{f}(\emph{\textbf{p}}^{T}\xi), \quad \xi\in \mathbb{R},\end{align}
where   $\widehat{g}(\gamma):=\int_{\mathbb{R}^{d}}g(\emph{\textbf{x}})e^{-\texttt{i}\emph{\textbf{x}}\cdot \gamma}d\emph{\textbf{x}}$ is the Fourier transform of any function  $g\in L^{p}(\mathbb{R}^{d}).$
It follows from \eqref{Radonform} that    $\widehat{\mathcal{R}_{\emph{\textbf{p}}}f}$ is essentially obtained by taking the cross-section of
$\widehat{f}$ on  the subspace (slice) $\{\emph{\textbf{p}}^{T}\xi: \xi\in \mathbb{R}\}$.

The  central problem of  CT is to use the  Radon transform   to reconstruct the target function $f$.
The most classical reconstruction approach is the filtered backprojection (FBP) (c.f.\cite{Natterer,Natterer1}).
It states that if   $f$ is bandlimited then it can be reconstructed via
\begin{align}\label{pppp}
f(x,y)&=\frac{1}{4\pi^{2}}\int^{2\pi}_{0}\int^{\infty}_{0}\widehat{\mathcal{R}_{\emph{\textbf{p}}}f}(\xi)e^{\texttt{i}\xi(x\cos\theta+y\sin\theta)}\xi d\xi d\theta,
\end{align}
where $\emph{\textbf{p}}=(\cos\theta, \sin\theta)$.


\subsection{Traditional reconstruction  approaches    conducted by Radon transform  at multiple angles and our single angle-based  problem}

Theoretically,    the reconstruction of $f$ via \eqref{pppp}  requires   the  cross-sections   $\widehat{\mathcal{R}_{\emph{\textbf{p}}}f}(\xi)=\widehat{f}(\emph{\textbf{p}}^{T}\xi)$
  for all angles   $\theta\in [0, 2\pi)$.
  In practice, however, what one  can observe are the   samples of a limited number of  cross-sections.
Therefore, the essential problem of CT  is to reconstruct  $f$
by the  samples of  finitely many   cross-sections.
Based on  \eqref{pppp}, many reconstruction algorithms have been designed (c.f. \cite{Shengli,Shengli2,Kak}).
Some recent  alternatives     to FBP have been introduced (e.g. \cite{Unser2,YXu}). Unlike FBP, they are conducted by the samples of
Radon transforms.  For example, based on the orthogonal polynomial system,  Xu
\cite{YXu} established the approach to CT.  McCann   and  Unser \cite{Unser2} established a    spline-based  reconstruction.

Note  that the  samples required  for   the above approaches    are derived from  Radon transforms at  multiple angles, and naturally   we confront
the following problem:
\begin{align} \label{SPPnew} \begin{array}{llll} \hbox{\textbf{Q}}: \hbox{Can  a function be exactly reconstructed  by  its Randon (transform) samples}\\
\quad \quad \hbox{ at a single angle (SA)}?\end{array}\end{align}
If such a single angle CT (SACT) problem  is feasible then    a function can be  determined by its Radon transform at
a SA. It is essentially the injectivity problem of Radon transform (c.f. S. Helgason  \cite{Helgason}).
Due to \eqref{pppp}, we do  not  anticipate     the injectivity   can be achieved
for any function in $L^{2}(\mathbb{R}^{2})$. Instead  it follows from    \cite{Helgason}
that  it can be achieved in some subspaces of  $L^{2}(\mathbb{R}^{2})$. While  there are some results on such an  injectivity
problem (e.g. \cite{DERIKVAND,FLY,Helgason,Homan,Ilmavirta}), the related sampling problem in \eqref{SPPnew} remains  less explored.
In what follows we briefly explain why such a sampling problem is significant from  the real-time imaging perspective.

\subsection{SACT is required  for real-time imaging}
\label{fastimaging}
Optical imaging has been widely used in observing biological objects, such as blood cells (thin objects) and bones (thick objects). The thin objects are commonly imaged directly by refractive-index distributions, which  is  achieved by holographic tomography (HT) (\cite{Kim}). However, for imaging thick  objects,   CT  is  usually employed.

CT   commonly  requires  samples  (measurements) of the light fields penetrating through the object    from different angles (views).
To do so, the object  needs to be rotated by a rotation motor (\cite{Yablon}) or the illumination needs to be scanned by a beam steering device, which not only causes instability for the imaging system, but makes the system bulky (\cite{Antipa,Horisaki}). More importantly, limited by the time of recording fields, rotating objects or scanning illuminations becomes not suitable for real-time imaging, especially for observing fast dynamic events (\cite{Horisaki}). Therefore this naturally leads to the  following  imaging  problem:

 \begin{align} \label{tywm}  \hbox{Under what condition can  CT
be achieved by  the samples of Radon transform at SA?}
\end{align}
Most recently,  R. Horisaki, K. Fujii, and J. Tanida  \cite{Horisaki} established a
SA method for HT  by  inserting a diffuser.
 Note that the samples used  in \cite{Horisaki} are  required to    contain  the diffraction information while
 the Radon samples for CT commonly do not contain (c.f. \cite[section 1]{Mller}). Here the diffraction of light waves
 at an  aperture is computed  by the Fresnel  integral
 $$U(x, y)=\frac{1}{\texttt{i}\lambda z}\int_{\mathbb{R}^{2}}U(x_{0}, y_{0})e
 ^{\texttt{i}\frac{\kappa}{2z}[(x-x_{0})^{2}+(y-y_{0})^{2}]}dx_{0}dy_{0},$$
 where $U(x_{0}, y_{0})$ is the  transmission field, $U(x, y)$ is the field on the view plane,
$z$ is the distance between
the  aperture and the view plane, and  $\lambda$ and $\kappa$ are the wavelength and wave number,
respectively.
 Therefore, the SA method in \cite{Horisaki}
is not applicable for CT.
To the best of our knowledge,  the theoretical study of sampling  problem \eqref{tywm} (or \eqref{SPPnew}) has not been fully explored yet  in optics.

\subsection{The SACT problem   in shift-invariant space (SIS)}
The   shift-invariant space (SIS) is  a   type of  function space that is  widely  applied    in
approximation theory, wavelet analysis and signal processing
(e.g. \cite{Fienup289,Aldroubi3,Zayed0,Chui,Feris,BHBOOK,Matla,SUNQIY1,yunzhang1}).
Throughout this paper, the SIS is denoted by
\begin{align}\label{SISdedingyi}V(\varphi, \mathbb{Z}^2)=\Big\{\sum_{\emph{\textbf{k}}\in \mathbb{Z}^{2}}c_{\emph{\textbf{k}}}\varphi(\cdot-\emph{\textbf{k}}): \sum_{\emph{\textbf{k}}\in \mathbb{Z}^{2}}|c_{\emph{\textbf{k}}}|^{2}<\infty\Big\},\end{align}
where $\varphi\in L^{2}(\mathbb{R}^2)$ is referred to  as the generator.
Correspondingly, our purpose here  is to examine the   SACT problem   \eqref{SPPnew}
in the SIS setting:
\begin{align} \label{SPP} \begin{array}{llll} \hbox{\textbf{Q}}: \hbox{How can  a compactly supported  function $f\in V(\varphi, \mathbb{Z}^2)$ be exactly reconstructed}\\
\quad \quad \hbox{by  its Randon (transform) samples at a single angle (SA)}?\end{array}\end{align}
\subsection{Assumption on the support of  target function, and definition of positive definite  function}
Before introducing our  main contributions, some  denotations are necessary.
Throughout this paper,
suppose that the generator  $\varphi\in  L^{2}(\mathbb{R}^{2})$ is compactly supported   such that
\begin{align}\label{KKK45}\hbox{supp}(\varphi)\subseteq[N_{1}, M_{1}]\times[N_{2}, M_{2}],\end{align}
and the shift system $\{\varphi(\cdot-\emph{\textbf{k}}): \emph{\textbf{k}}\in \mathbb{Z}^2\}$
is linearly independent in $L^{2}(\mathbb{R}^{2})$.
Moreover, the arbitrary   target function   $f\in V(\varphi, \mathbb{Z}^2)$ is compactly supported
such that \begin{align}\label{KKK46} \hbox{supp}(f)\subseteq [a_{1}, b_{1}]\times [a_{2}, b_{2}].\end{align}
By \eqref{KKK45} and \eqref{KKK46}, there exists a finite  sequence  $\{c_{\emph{\textbf{k}}_{l}}, l=1, \ldots, \#E\}\subseteq \mathbb{C}$ such that  $f$ can be expressed as
\begin{align}\label{BVZXC1234} f=\sum^{\#E}_{l=1}c_{\emph{\textbf{k}}_{l}}\varphi(\cdot-\emph{\textbf{k}}_{l}),\end{align}
where $ E:=\{\emph{\textbf{k}}_{1}, \ldots, \emph{\textbf{k}}_{\#E}\}=\big\{\big[\lceil a_{1}-M_{1}\rceil, \lfloor b_{1}-N_{1}\rfloor\big]\times
\big[\lceil a_{2}-M_{2}\rceil, \lfloor b_{2}-N_{2}\rfloor\big]\big\}\cap \mathbb{Z}^{2}
$,
$\#E$ is the cardinality of $E$, and
$\lceil x\rceil$ ($\lfloor x\rfloor$)
is  the smallest (largest)  integer  that is not smaller (larger)  than $x\in \mathbb{R}$, respectively.
In what follows we explain that the assumption in \eqref{KKK46} is reasonable.

\begin{rem}\label{support} Throughout this paper, as in \eqref{KKK46}  we assume that  the function  to be reconstructed  is  compactly
 supported and its supports is  contained in a known rectangle.
 Such an   assumption     is reasonable  for CT (e.g. \cite{YXu}) since   from the optical  perspective, the function to be reconstructed in CT is the difference between the  refractive index distribution of the object and that of the surrounding medium  (c.f. \cite{Mller}), and consequently it is generally   compactly supported. Moreover, the support of the function is known  when   the boundary of the
 object is clear.
\end{rem}

In what follows we recall the   definition of  positive definite     functions
which have  been  extensively   applied to scattered data interpolation,  approximation theory and harmonic analysis (e.g. \cite{QSIS4,Hinrichs,Karimi,Wendlan}).
We say that a function $\phi: \mathbb{R}^{d}\longrightarrow \mathbb{C}$   is positive definite  if  for all $N\in \mathbb{N}$, all sets $X=\{\emph{\textbf{x}}_{1}, \emph{\textbf{x}}_{2}, \ldots, \emph{\textbf{x}}_{N}\}\subseteq \mathbb{R}^{d}$, and all vectors  $\textbf{0}\neq (\alpha_{1}, \ldots, \alpha_{N})^{T}\in \mathbb{C}^{N}$, the quadratic form $\sum^{N}_{j=1}\sum^{N}_{k=1}\alpha_{j}\overline{\alpha}_{k}\phi(\emph{\textbf{x}}_{j}-\emph{\textbf{x}}_{k})>0.$
We will  recall   more properties of  positive definite functions in subsection \ref{moreqe}.

\subsection{Main contributions and their common features}
There are five main results in this paper. They  will be established in
subsections \ref{1stresult}, \ref{2ndresult}, \ref{3rdresult}, \ref{neirong2} and
\ref{vanishingqingkuang}. From the perspective of the properties satisfying by the generator $\varphi$,
these  main results  are organized briefly as follows.

$\bullet$ \textbf{The nonvanishing case} ($\widehat{\varphi}(\textbf{0})=\int_{\mathbb{R}^{2}}\varphi(\emph{\textbf{x}})d\emph{\textbf{x}}\neq0$).
The  set $\Lambda$ of eligible   direction vectors is constructed  for the  SACT of any
$f\in V(\varphi, \mathbb{Z}^2)$ satisfying  \eqref{KKK46}. It is  proved that  for any  $\emph{\textbf{p}}\in
\Lambda$,  there exists a sampling set $X_{\emph{\textbf{p}}}\subseteq \mathbb{R}$ (having the cardinality $\#E$)
such that   $f $   can be determined uniquely by its SA Radon samples at $X_{\emph{\textbf{p}}}$,
where the set $E$ is correlated with  $f$ via \eqref{BVZXC1234}.
Additionally, if $\varphi\in C^{1}(\mathbb{R}^{2})$ then $X_{\emph{\textbf{p}}}$
is   constructed explicitly.

$\bullet$ \textbf{The vanishing case} ($\widehat{\varphi}(\textbf{0})=\int_{\mathbb{R}^{2}}\varphi(\emph{\textbf{x}})d\emph{\textbf{x}}=0$).
As in  the nonvanishing case,      the set $\Omega$ of eligible   direction vectors is constructed
 for the SACT of any   $f\in V(\varphi, \mathbb{Z}^2)$ satisfying \eqref{KKK46}. The set $\Omega$ is different from the above $\Lambda$
 in the nonvanishing case.  For any $\emph{\textbf{p}}\in \Omega$,  the existence of the  eligible  sampling set $X_{\emph{\textbf{p}}}\subseteq \mathbb{R}$ (also having the cardinality $\#E$)  is proved.
 Consequently,
 $f$   can be determined uniquely by its SA Radon samples at $X_{\emph{\textbf{p}}}$.
 Additionally, for the case that  $\varphi\in C^{1}(\mathbb{R}^{2})$  the sampling set     $X_{\emph{\textbf{p}}}$
is constructed explicitly.

 $\bullet$ \textbf{The positive definite generator case}.   Suppose  that  $\varphi$ is positive definite.
 Eligible direction vector sets are constructed  for the nonvanishing and vanishing cases, respectively.
 For any eligible direction vector $\emph{\textbf{p}}$, the target function  $f\in V(\varphi, \mathbb{Z}^2)$ satisfying \eqref{KKK46} can be determined uniquely by its SA samples
 at $\{\emph{\textbf{p}}\emph{\textbf{k}}_{1}, \ldots, \emph{\textbf{p}}\emph{\textbf{k}}_{\#E}\}$,
 where $\{\emph{\textbf{k}}_{1}, \ldots, \emph{\textbf{k}}_{\#E}\}=E$
 is correlated with $f$ via \eqref{BVZXC1234}.

\begin{rem}\label{CVBNM}
There are two common features of the above three main  contributions. (1) The samples for CT are derived from the
SA Radon transform but not from  multiple-angle Radon transforms. (2) Note that $\#E$  Radon samples are used to
determined $f$.
Recall again that $\{\varphi(\cdot-\emph{\textbf{k}}): \emph{\textbf{k}}\in \mathbb{Z}^2\}$
is linearly independent, and
by
\eqref{BVZXC1234},  $f=\sum^{\#E}_{l=1}c_{\emph{\textbf{k}}_{l}}\varphi(\cdot-\emph{\textbf{k}}_{l})$.
Then  $f$ is determined uniquely  by the  $\#E$ coefficients: $c_{\emph{\textbf{k}}_{1}}, \ldots,
c_{\emph{\textbf{k}}_{\#E}}$.  Therefore we only use the \textbf{minimum} number of samples in our SA-based
reconstruction.
\end{rem}

\subsection{Outline of the paper}
 In Theorem \ref{main1}  a sufficient and necessary condition is
 established on the pair $(\varphi, \emph{\textbf{p}})$ such that, an arbitrary  compactly supported target function  $f\in V(\varphi, \mathbb{Z}^2)$
 satisfying \eqref{KKK46} can be determined uniquely by its SA Radon transform $\mathcal{R}_{\emph{\textbf{p}}}f$.
 With the help of Paley-Wiener theorem, it will be explained  in subsection \ref{bupingfan} that
 such a determination problem is absolutely nontrivial.
 Based on Theorem \ref{main1} we will  address the SACT sampling  problem  \eqref{SPP} in section \ref{boxspline}
 and section \ref{moreq23}.

 Section \ref{boxspline} concerns on the  problem \eqref{SPP} for compactly supported functions in $V(\varphi, \mathbb{Z}^2)$
 where $\varphi$ is a general generator.
   Theorem \ref{theorem123} establishes  a sufficient and necessary    condition  on  $(\varphi, \emph{\textbf{p}}, X_{\emph{\textbf{p}}})$  such that the SACT sampling \eqref{SPP} can be achieved by the SA Radon samples at   $X_{\emph{\textbf{p}}}$.
 For the  general generator $\varphi$ case, a natural problem is the existence of $\emph{\textbf{p}}$
 and $X_{\emph{\textbf{p}}}$.
The answer to this problem will be addressed in  Theorem  \ref{yibanshengcyuan}  for the nonvanshing ($\widehat{\varphi}(\textbf{0})\neq0$) case and in Theorem  \ref{yibanshengcyuanYY} for the vanishing ($\widehat{\varphi}(\textbf{0})=0$) case, where a set of eligible  direction vectors $\Lambda$ (respectively,  $\Omega$) is provided in Theorem  \ref{yibanshengcyuan} (respectively,  Theorem  \ref{yibanshengcyuanYY}) such that for any
$\emph{\textbf{p}}\in \Lambda$ (or $\emph{\textbf{p}}\in \Omega$) there exists  a sampling set $X_{\emph{\textbf{p}}}$,
and consequently $f$ can be determined uniquely by its SA Radon  samples
at $X_{\emph{\textbf{p}}}.$ In particular,  an explicit construction of a sampling set $X_{\emph{\textbf{p}}}$ was presented in   Theorem \ref{xianshibiaoshi}
and Proposition \ref{hbvcxz} for the case when $\varphi\in C^{1}(\mathbb{R}^{2})$.


The purpose of Section \ref{moreq23} is to address the condition on $(\varphi, \emph{\textbf{p}})$
such that the  compactly supported  $f\in V(\varphi, \mathbb{Z}^2)$ satisfying \eqref{KKK46} can be determined uniquely by its SA Radon samples at  $\{\emph{\textbf{p}}\emph{\textbf{k}}_{1}, \ldots, \emph{\textbf{p}}\emph{\textbf{k}}_{\#E}\}$,
 where $\{\emph{\textbf{k}}_{1}, \ldots, \emph{\textbf{k}}_{\#E}\}=E$.
Such a  condition is established in Theorem \ref{HHKK}.
Based on Theorem \ref{HHKK}, we address the case that $\varphi$ is positive definite in
Theorems \ref{ddgo} and \ref{eligiblevectors}. In  particular,  Theorem \ref{ddgo} applies to the nonvanishing case
while Theorem \ref{eligiblevectors} applies to the vanishing case.

 %
%


\section{Preliminary}

\subsection{On the  support of  $\mathcal{R}_{\emph{\textbf{p}}}f$}
For a function $f\in  L^{1}(\mathbb{R}^{2})$ and a  direction  vector  $\emph{\textbf{p}}=(\cos\theta, \sin\theta)$,
 motivated by  \cite{Han1,BHBOOK} we next address the relationship  between     $\mathcal{R}_{\emph{\textbf{p}}}f$ and $f$ in the  spatial  domain. Denote the singular value decomposition (SVD) of $\emph{\textbf{p}}$
by $\emph{\textbf{p}}=\Sigma V^{T}$, where  $V$ is a  $2\times 2$ real-valued unitary
matrix  and $\Sigma=(
1, 0)$.
Now it follows from     \cite{Han1,BHBOOK} that
\begin{align}\label{HVCXZ}
\mathcal{R}_{\emph{\textbf{p}}}f=\Sigma(V^{T}f),
\end{align}
where $V^{T}f(\emph{\textbf{x}})=f((V^{T})^{-1}\emph{\textbf{x}})$ with  $\emph{\textbf{x}}=(x_{1},   x_{2})^{T}\in \mathbb{R}^{2}$,  and for any $g$ on $\mathbb{R}^{2}$ the function  $\Sigma g$ on $\mathbb{R}$ is defined by
\begin{align}\label{cvb}
\Sigma g (x_{1})=\int_{\mathbb{R}} g(x_{1},
x_{2})dx_{2}.
\end{align}
The following remark is derived  from \cite[section 1]{Han1}.
\begin{rem}\label{JHR}
If $f\in L^2(\mathbb{R}^{2})$  is compactly supported then its Radon transform $\mathcal{R}_{\emph{\textbf{p}}}f$
can be  expressed as $\Sigma(V^{T}f)$ in \eqref{HVCXZ}.
\end{rem}

It has been stated  in subsection \ref{subsection1.1} that
if $f\in
L^2(\mathbb{R}^2)$ is compactly supported then $\mathcal{R}_{\emph{\textbf{p}}}f\in L^{2}(\mathbb{R})$. We include its proof together with support information in the following lemma.


\begin{lem}\label{pouy}
Suppose that $f\in L^{2}(\mathbb{R}^{2})$  with   $\hbox{supp}(f)\subseteq [a_{1}, b_{1}]\times [a_{2}, b_{2}]$.
Then
\begin{align}\label{qujian} \hbox{supp}(\mathcal{R}_{\emph{\textbf{p}}}f)\subseteq [-\sqrt{2}\max\{|b_{i}|, |a_{i}|: i=1, 2\}, \sqrt{2}\max\{|b_{i}|, |a_{i}|: i=1, 2\}],\end{align}
and $\mathcal{R}_{\emph{\textbf{p}}}f\in L^{2}(\mathbb{R})$.
\end{lem}
\begin{proof}
Let $V$ be the real unitary matrix from the SVD of $\emph{\textbf{p}}$ such that
$\emph{\textbf{p}}=\Sigma V^{T}$ with $\Sigma=(
1, 0)$.
Denote $V^{T}f$ by $g.$ Then for any $x_{1}\in \mathbb{R},$
 we have
\begin{align}\label{en123456}
\begin{array}{llll}
\mathcal{R}_{\emph{\textbf{p}}}f(x_{1})&=\displaystyle\Sigma g(x_{1})\\
&\displaystyle=\int_{\mathbb{R}} g(x_{1},
x_{2})dx_{2}\\
&\displaystyle=\int_{\mathbb{R}} [V^{T}f](x_{1},
x_{2})dx_{2}\\
&\displaystyle=\int_{\mathbb{R}} f(V(x_{1},
x_{2})^{T})dx_{2}, \ (\ref{en123456} A)
\end{array}
\end{align}
where  the first and  second  equalities   are  derived from Remark \ref{JHR}  and  \eqref{cvb}, respectively.
It follows from   $\hbox{supp}(f)\subseteq [a_{1}, b_{1}]\times [a_{2}, b_{2}]$ that  for any $\emph{\textbf{x}}\in \hbox{supp}(f)$, we have
\begin{align}\label{fvxczx}\|\emph{\textbf{x}}\|_{2}\leq\sqrt{2}\max\{|b_{i}|, |a_{i}|: i=1, 2\}.\end{align}
It follows from   \eqref{en123456} and the fact that
$V$ is  a unitary matrix, we have  $|x_{1}|\leq\sqrt{2}\max\{|b_{i}|, |a_{i}|: i=1, 2\}.$  Then   \eqref{qujian} holds.

Define $G(x_{1}, x_{2}):=f(V(x_{1}, x_{2})^{T})$. By \eqref{fvxczx} and $V$ being  a real unitary matrix,  we have
\begin{align}\label{uytt} |x_{2}|\leq
\sqrt{2}\max\{|b_{i}|, |a_{i}|: i=1, 2\}\end{align}
for any $(x_{1}, x_{2})^{T}\in \hbox{supp}(G)$.
Moreover,
\begin{align}\label{en123}
\begin{array}{llll}
\|\mathcal{R}_{\emph{\textbf{p}}}f\|_{L^{2}(\mathbb{R})}^{2}&=\|\Sigma(V^{T}f)\|_{L^{2}(\mathbb{R})}^{2}\\
&\displaystyle=\int_{\mathbb{R}}\big|\int_{\mathbb{R}} [V^{T}f](x_{1},
x_{2})dx_{2}\big|^{2} dx_{1}\\
&\displaystyle\leq\sqrt{2}\max\{|b_{i}|, |a_{i}|: i=1, 2\}\int_{\mathbb{R}}\int_{\mathbb{R}} |[V^{T}f](x_{1},
x_{2})|^{2}dx_{2}dx_{1}\\
&\displaystyle=\sqrt{2}\max\{|b_{i}|, |a_{i}|: i=1, 2\}\|f\|_{L^{2}(\mathbb{R}^2)}^{2}\\
&<\infty,
\end{array}
\end{align}
where the first inequality is derived from \eqref{uytt} and  the Cauchy-Schwarz inequality.
This completes the proof.
\end{proof}

\subsection{(Quasi) Shift-invariant space}\label{Rieszbound}
For a generator    $\varphi\in  L^{2}(\mathbb{R}^{2})$, as in \eqref{SISdedingyi}
its  associated  shift-invariant space (SIS) $V(\varphi, \mathbb{Z}^{2})$ is defined to be
\begin{align}\label{SISDE} V(\varphi, \mathbb{Z}^{2}):=\Big\{\sum_{\emph{\textbf{k}}\in \mathbb{Z}^{2}}c_{\emph{\textbf{k}}}\varphi(\cdot-\emph{\textbf{k}}): \{c_{\emph{\textbf{k}}}\}_{\emph{\textbf{k}}\in \mathbb{Z}^{2}}\in \ell^{2}(\mathbb{Z}^{2})\Big\},\end{align}
where $\ell^{2}(\mathbb{Z}^2)$ is the space of square summable sequences
such that any  $\{c_{\emph{\textbf{k}}}\}_{\emph{\textbf{k}}\in \mathbb{Z}^{2}}\in \ell^{2}(\mathbb{Z}^{2})$
satisfies $\|\{c_{\emph{\textbf{k}}}\}_{\emph{\textbf{k}}\in \mathbb{Z}^{2}}\|_{\ell^{2}(\mathbb{Z}^2)}=(\sum_{\emph{\textbf{k}}\in \mathbb{Z}^2}|c_{\emph{\textbf{k}}}|^{2})^{1/2}<\infty.$
As mentioned in section \ref{diyijiesection}, throughout the paper  the system
$\{\varphi(\cdot-\emph{\textbf{k}}): \emph{\textbf{k}}\in \mathbb{Z}^2\}$ is required to be linearly independent in $L^2(\mathbb{R}^2)$.
A sufficient condition for the linear independence is that $\{\varphi(\cdot-\emph{\textbf{k}}): \emph{\textbf{k}}\in \mathbb{Z}^2\}$
satisfies the so called Riesz basis condition, namely, there exist constants  $0<C_{1}\leq C_{2}<\infty$
such that for any  $\{c_{\emph{\textbf{k}}}\}_{\emph{\textbf{k}}\in \mathbb{Z}^2}\in \ell^{2}(\mathbb{Z}^2)$ there holds
\begin{align}\label{chuan} C_{1}\|\{c_{\emph{\textbf{k}}}\}_{\emph{\textbf{k}}\in \mathbb{Z}^{2}}\|^{2}_{\ell^{2}(\mathbb{Z}^2)}\leq\big\|\sum_{\emph{\textbf{k}}\in \mathbb{Z}^{2}}c_{\emph{\textbf{k}}}\varphi(\cdot-\emph{\textbf{k}})\big\|^{2}_{L^{2}(\mathbb{R}^{2})}\leq C_{2}\|\{c_{\emph{\textbf{k}}}\}_{\emph{\textbf{k}}\in \mathbb{Z}^{2}}\|^{2}_{\ell^{2}(\mathbb{Z}^2)}.\end{align}

For a generator $\varphi\in  L^{2}(\mathbb{R}^{2})$ and the shift set  $\mathcal{X}=\{x_{\emph{\textbf{k}}}\}_{\emph{\textbf{k}}\in \mathbb{Z}^{2}}\subseteq \mathbb{R}^{2}$, its  associated quasi shift-invariant space  (QSIS) is defined as
\begin{align}\label{QSIS}V(\varphi, \mathcal{X}):=\Big\{\sum_{\emph{\textbf{k}}\in \mathbb{Z}^{2}}c_{\emph{\textbf{k}}}\varphi(\cdot-x_{\emph{\textbf{k}}}): \{c_{\emph{\textbf{k}}}\}_{\emph{\textbf{k}}\in \mathbb{Z}^{2}}\in \ell^{2}(\mathbb{Z}^{2})\Big\}.\end{align}
If $\mathcal{X}=\mathbb{Z}^{2}$ then $V(\varphi, \mathcal{X})$ degenerates to a SIS.
As implied  in \cite{QSIS4}, the recovery    for the functions in   $ V(\varphi, \mathcal{X})$ ($\mathcal{X}\neq\mathbb{Z}^{2}$)
is much more complicated than  that for  the SIS. For such a recovery,   by
\cite[section 3.1(A1)]{QSIS4} it is required that  $\varphi$ is    positive definite.



\subsection{Sobolev smoothness of a function}\label{sobolevsmoothness}
For any $\varsigma\in\mathbb{R}$,  the Sobolev space $H^{\varsigma}(\mathbb{R}^{d})$ (c.f.\cite{BHS,LHYH,Youfa})
is defined as
\begin{align}\begin{array}{lllll}\label{defi_s}\displaystyle  H^{\varsigma}(\mathbb{R}^{d}):=\Big\{f: \int_{\mathbb{R}^{d}}|\widehat{f}(\xi)|^{2} (1+\|\xi\|_{2}^{2})^{\varsigma}d\xi<\infty\Big\}.\end{array}\end{align}
Clearly, if $\varsigma\geq0$ then $H^{\varsigma}(\mathbb{R}^{d})\subseteq L^2(\mathbb{R}^d)$.
 The deduced norm is defined   by
\begin{align}\notag \begin{array}{lllll} \displaystyle \|f\|_{H^{\varsigma}(\mathbb{R}^{d})}:=\frac{1}{(2\pi)^{d/2}}\Big(\int_{\mathbb{R}^{d}}|\widehat{f}(\xi)|^{2}(1+\|\xi\|_{2}^{2})^{\varsigma}d\xi\Big)^{1/2}, \quad \forall f\in H^{\varsigma}(\mathbb{R}^{d}).\end{array}\end{align}
The Sobolev smoothness of $f$ is defined as $\nu_{2}(f):=\sup\{\varsigma: f\in H^{\varsigma}(\mathbb{R}^{d})\}$.
The following lemma is derived from \cite[Lemma 2.4]{Hanprojectionold}. It states that for a compactly supported $f\in L^2(\mathbb{R}^2)$,
the Sobolev smoothness  of $\mathcal{R}_{\emph{\textbf{p}}}f$ is not smaller  than  $\nu_{2}(f)$.

\begin{lem} \label{fff987f}
Suppose that $f\in H^{\varsigma}(\mathbb{R}^{2}), \varsigma\geq0$ is compactly supported.
Then $\nu_{2}(\mathcal{R}_{\emph{\textbf{p}}}f)\geq \nu_{2}(f)$ for any direction vector $\emph{\textbf{p}}.$
\end{lem}

With the help of Lemma \ref{fff987f} we next address the continuity of the Radon transform.

\begin{prop}\label{lianxuxing}
Suppose that $f\in H^{\varsigma}(\mathbb{R}^{d})$ such that $\varsigma>d/2$. Then
we have

(1) $f$ is   continuous.
\newline
Suppose that  $g\in H^{s}(\mathbb{R}^{2})$
with $s>1/2$ is compactly supported. Then we have

(2) the  Radon transform $\mathcal{R}_{\emph{\textbf{p}}}g$ is continuous for any direction vector  $\emph{\textbf{p}}.$
\end{prop}
\begin{proof}
For any $\emph{\textbf{x}}$ and $ \Delta \emph{\textbf{x}}\in \mathbb{R}^{d}$, we estimate
\begin{align}\label{ff1} \begin{array}{lllll} \displaystyle
|f(\emph{\textbf{x}}+\Delta \emph{\textbf{x}})-f(\emph{\textbf{x}})|\\
=\displaystyle\Big|\frac{1}{(2\pi)^{d/2}}\int_{\mathbb{R}^{d}}\widehat{f}(\xi)(e^{\texttt{i}(\emph{\textbf{x}}+\Delta \emph{\textbf{x}})\cdot\xi}-e^{\texttt{i}\emph{\textbf{x}}\cdot\xi})d\xi\Big|\\
\leq\displaystyle\frac{2}{(2\pi)^{d/2}}\int_{\|\xi\|_{2}\leq1}|\widehat{f}(\xi)||\sin(\xi\cdot\Delta \emph{\textbf{x}}/2)| d\xi+\frac{2}{(2\pi)^{d/2}}\displaystyle\int_{\|\xi\|_{2}>1}|\widehat{f}(\xi)||\sin(\xi\cdot\Delta \emph{\textbf{x}}/2)| d\xi\\
=I_{1}(\Delta \emph{\textbf{x}})+I_{2}(\Delta \emph{\textbf{x}}),
\end{array}\end{align}
where $(e^{\texttt{i}(\emph{\textbf{x}}+\Delta \emph{\textbf{x}})\cdot\xi}-e^{\texttt{i}\emph{\textbf{x}}\cdot\xi})=2\texttt{i}e^{\texttt{i}\emph{\textbf{x}}\cdot\xi}e^{\texttt{i}\Delta \emph{\textbf{x}}\cdot\xi/2}\sin(\Delta \emph{\textbf{x}}\cdot\xi/2)$
is used in the inequality.
For  $\|\xi\|_{2}\leq1$,  it follows from   $|\sin(\xi\cdot\Delta \emph{\textbf{x}}/2)|\leq |\xi\cdot\Delta \emph{\textbf{x}}/2|\leq \frac{1}{2}\|\xi\|_{2}\|\Delta \emph{\textbf{x}}\|_{2}\leq \frac{1}{2}\|\Delta \emph{\textbf{x}}\|_{2}$ that
\begin{align}\label{ff2} \begin{array}{lllll}    I_{1}(\Delta \emph{\textbf{x}})&\displaystyle\leq \|\Delta \emph{\textbf{x}}\|_{2}\frac{1}{(2\pi)^{d/2}}\int_{\|\xi\|_{2}\leq1}|\widehat{f}(\xi)| d\xi
\\
&\displaystyle\leq\|\Delta \emph{\textbf{x}}\|_{2}\|f\|_{H^{\varsigma}(\mathbb{R}^d)}.\end{array}\end{align}
On the other hand, we  choose $\mu\in (0, 1)$ such that $\varsigma-\mu>d/2$. Then
\begin{align}\label{ff3} \begin{array}{lllll}    I_{2}(\Delta \emph{\textbf{x}})&\displaystyle\leq \frac{1}{(2\pi)^{d/2}}(\frac{\|\Delta \emph{\textbf{x}}\|_{2}}{2})^{\mu}\int_{\|\xi\|_{2}\geq 1}|\widehat{f}(\xi)|\|\xi\|^{\mu}_{2} d\xi\\
&\displaystyle\leq \frac{1}{(2\pi)^{d/2}} (\frac{\|\Delta \emph{\textbf{x}}\|_{2}}{2})^{\mu}\int_{\|\xi\|_{2}\geq 1}|\widehat{f}(\xi)|(1+\|\xi\|^{2}_{2})^{\varsigma/2}
(1+\|\xi\|^{2}_{2})^{-\varsigma/2}\|\xi\|^{\mu}_{2} d\xi\\
&\displaystyle\leq  (\frac{\|\Delta \emph{\textbf{x}}\|_{2}}{2})^{\mu}\|f\|_{H^{\varsigma}(\mathbb{R}^d)}
\int_{\|\xi\|_{2}\geq 1}(1+\|\xi\|^{2}_{2})^{-(\varsigma-\mu)}d\xi\end{array},\end{align}
where the first inequality is derived from
$$|\sin(\xi\cdot\Delta \emph{\textbf{x}}/2)|\leq |\sin(\xi\cdot\Delta \emph{\textbf{x}}/2)|^{\mu}\leq |\xi\cdot\Delta \emph{\textbf{x}}/2|^{\mu}\leq (\frac{1}{2}\|\xi\|_{2}\|\Delta \emph{\textbf{x}}\|_{2})^{\mu},$$
and  the last inequality is   from the Cauchy-Schwarz inequality.
Combining \eqref{ff1}, \eqref{ff2} and \eqref{ff3} we have
$$|f(\emph{\textbf{x}}+\Delta \emph{\textbf{x}})-f(\emph{\textbf{x}})|\leq\Theta_{f,\mu}\max\{\|\Delta \emph{\textbf{x}}\|_{2}, \|\Delta \emph{\textbf{x}}\|^{\mu}_{2}\},$$
where $$\Theta_{f,\mu}=\|f\|_{H^{\varsigma}(\mathbb{R}^d)}+\frac{1}{2^{\mu}}\|f\|_{H^{\varsigma}(\mathbb{R}^d)}
\int_{\|\xi\|_{2}\geq 1}(1+\|\xi\|^{2}_{2})^{-(\varsigma-\mu)}d\xi<\infty.$$ This
proves the first part of the proposition.

For any compactly supported  $g\in H^{s}(\mathbb{R}^{2})$
with $s>1/2$, by Lemma \ref{fff987f} the Sobolev smoothness $\nu_{2}(\mathcal{R}_{\emph{\textbf{p}}}g)\geq \nu_{2}(g)>1/2$.
By the first part of the present  proposition, $\mathcal{R}_{\emph{\textbf{p}}}g$ is continuous. The proof is concluded.
\end{proof}

\begin{rem}\label{continuous123}
For $f\in H^{\varsigma}(\mathbb{R}^d)$ with $\varsigma>d/2$, by Proposition \ref{lianxuxing} (1)
we have $f\in C(\mathbb{R}^d)$. But   it does not necessarily implies that $f\in C^{1}(\mathbb{R}^d)$.
For example, define $f(x_{1}, x_{2})=[\chi_{(0,1]}\star \chi_{(0,1]}](x_{1})[\chi_{(0,1]}\star \chi_{(0,1]}](x_{2})$,
where $\star $ is the convolution and $\chi_{(0,1]}$ is the characteristic function on  the interval   $(0,1].$
By direct calculation we have
\begin{align}\label{continuousert}[\chi_{(0,1]}\star \chi_{(0,1]}](x_{j})=\left\{\begin{array}{cccccccccc}
x_{j},&0<x_{j}\leq1,\\
2-x_{j}, &1<x_{j}\leq2,\\
0, &\hbox{else}.
\end{array}\right.\end{align}
On the other hand, one can check that
\begin{align}\label{KKKKGT}\widehat{\chi_{(0,1]}}(\xi_{j})=e^{-\texttt{i} \xi_{j}/2}\frac{\sin\xi_{j}/2}{\xi_{j}/2}.\end{align} Therefore, $$\widehat{f}(\xi_{1}, \xi_{2})=
e^{-\texttt{i} \xi_{1}}[\frac{\sin\xi_{1}/2}{\xi_{1}/2}]^2e^{-\texttt{i} \xi_{2}}[\frac{\sin\xi_{2}/2}{\xi_{2}/2}]^{2}.$$
From this, one  can check that $f\in H^{\varsigma}(\mathbb{R}^2)$ for any $\varsigma<3/2$. But it is clear from \eqref{continuousert}  that   $f\notin C^{1}(\mathbb{R}^{2})$.
\end{rem}

\begin{rem}\label{kenenglianxu}
The purpose  here is to state  that there exist   functions which are discontinuous but their
Radon transforms are  continuous.
For example, define $f(x_{1}, x_{2})=\chi_{(0,1]}(x_{1})\chi_{(0,1]}(x_{2})$. It is clear that $f$ is discontinuous.
It follows from \eqref{KKKKGT} that $\widehat{f}(\xi_{1}, \xi_{2})=e^{-\texttt{i} \xi_{1}/2}\frac{\sin\xi_{1}/2}{\xi_{1}/2}
e^{-\texttt{i} \xi_{2}/2}\frac{\sin\xi_{2}/2}{\xi_{2}/2}$.
Now for any $\emph{\textbf{p}}=(\cos\theta, \sin\theta)$ such that $\cos\theta \sin\theta$
$\neq0$, we have
$$\widehat{\mathcal{R}_{\emph{\textbf{p}}}f}(\xi)=\widehat{f}(\xi\cos\theta, \xi\sin\theta)=
e^{-\texttt{i} \frac{\xi\cos\theta}{2}}\frac{\sin\frac{\xi\cos\theta}{2}}{\frac{\xi\cos\theta}{2}}
e^{-\texttt{i} \frac{\xi\sin\theta}{2}}\frac{\sin\frac{\xi\sin\theta}{2}}{\frac{\xi\sin\theta}{2}}.$$
For   $|\xi|>1$, $|\widehat{\mathcal{R}_{\emph{\textbf{p}}}f}(\xi)|\leq\frac{|\frac{2}{\cos \theta}||\frac{2}{\sin\theta}|}{\xi^2}$. From this and the continuity of $\widehat{\mathcal{R}_{\emph{\textbf{p}}}f}$, one can prove  that the Sobolev smoothness  $\nu_{2}(\mathcal{R}_{\emph{\textbf{p}}}f)>1/2.$
By   Proposition \ref{lianxuxing} (1), $\mathcal{R}_{\emph{\textbf{p}}}f$ is continuous.
\end{rem}


\section{A necessary and sufficient condition for the SA Radon transform-based  determination }
The following establishes a necessary and sufficient condition on the pair $(\varphi, \emph{\textbf{p}})$
 such that any  compactly supported function  $f\in V(\varphi, \mathbb{Z}^{2})$ can be determined by its
SA Radon transform $\mathcal{R}_{\emph{\textbf{p}}}f$. Although  such  a determination     depends on $\mathcal{R}_{\emph{\textbf{p}}}f$ and does not  use    its samples directly,
it will be helpful for answering the SACT  sampling problem  \eqref{SPP}.
As previously,  the vectors in $\mathbb{R}^{d}$  are considered as column vectors,
while  the direction vector $\emph{\textbf{p}}$ is a row vector.
\subsection{Determination result}
The following is the main result of the present  section.
\begin{theo}\label{main1}
Suppose that  $\varphi\in  L^{2}(\mathbb{R}^{2})$   such that    $\hbox{supp}(\varphi)\subseteq[N_{1}, M_{1}]\times
[N_{2}, M_{2}]$. Moreover, $\{\varphi(\cdot-\emph{\textbf{k}}): \emph{\textbf{k}}\in \mathbb{Z}^{2}\}$ is linearly independent
in $L^2(\mathbb{R}^2)$.
Then  any $f\in V(\varphi, \mathbb{Z}^{2})$ such that $\hbox{supp}(f)\subseteq [a_{1}, b_{1}]\times [a_{2}, b_{2}]$ can be determined uniquely  by its SA Radon transform   $\mathcal{R}_{\emph{\textbf{p}}}f$ if and only if
$\{\mathcal{R}_{\emph{\textbf{p}}}\varphi(\cdot-\emph{\textbf{p}}\emph{\textbf{k}}): \emph{\textbf{k}}\in E\}$
is linearly independent in $L^{2}(\mathbb{R})$, where
$ E=\big\{\big[\lceil a_{1}-M_{1}\rceil, \lfloor b_{1}-N_{1}\rfloor\big]\times
\big[\lceil a_{2}-M_{2}\rceil, \lfloor b_{2}-N_{2}\rfloor\big]\big\}\cap \mathbb{Z}^{2}.$
\end{theo}
\begin{proof}
We first prove that  $\mathcal{R}_{\emph{\textbf{p}}}(\varphi(\cdot-\emph{\textbf{k}}))=
\mathcal{R}_{\emph{\textbf{p}}}\varphi(\cdot-\emph{\textbf{p}}\emph{\textbf{k}})$
for any $\emph{\textbf{k}}\in \mathbb{Z}^2.$
Actually,   the Fourier transform of
 $\varphi(\cdot-\emph{\textbf{k}})$ at $\emph{\textbf{x}}\in \mathbb{R}^{2}$
 is $e^{-\texttt{i}\emph{\textbf{k}}\cdot\emph{\textbf{x}}}\widehat{\varphi}(\emph{\textbf{x}})$.
 Then by the Radon transform representation  \eqref{Radonform} in the  Fourier domain,  the Fourier transform  of the Radon transform  $\mathcal{R}_{\emph{\textbf{p}}}(\varphi(\cdot-\emph{\textbf{k}}))$ at $\xi\in \mathbb{R}$
is $e^{-\texttt{i}\emph{\textbf{k}}\cdot \emph{\textbf{p}}^{T}\xi}\widehat{\varphi}(\emph{\textbf{p}}^{T}\xi)$. Clearly, \begin{align}\label{40567}e^{-\texttt{i}\emph{\textbf{k}}\cdot \emph{\textbf{p}}^{T}\xi}\widehat{\varphi}(\emph{\textbf{p}}^{T}\xi)=e^{-\texttt{i}\emph{\textbf{p}}\emph{\textbf{k}}\xi}\widehat{\mathcal{R}_{\emph{\textbf{p}}}\varphi}(\xi).
\end{align}
Stated another way,
\begin{align}\label{890756}\mathcal{R}_{\emph{\textbf{p}}}(\varphi(\cdot-\emph{\textbf{k}}))=
\mathcal{R}_{\emph{\textbf{p}}}\varphi(\cdot-\emph{\textbf{p}}\textbf{\emph{\textbf{k}}}).\end{align}
Since $\varphi\in L^2(\mathbb{R}^{2})$  is compactly supported, then it follows from
 Lemma \ref{pouy} that $\mathcal{R}_{\emph{\textbf{p}}}(\varphi)\in L^2(\mathbb{R})$. Consequently,
 $\mathcal{R}_{\emph{\textbf{p}}}\varphi(\cdot-\emph{\textbf{p}}\textbf{\emph{\textbf{k}}})\in L^{2}(\mathbb{R})$
 for any $\emph{\textbf{k}}\in E$.

For convenient narration, denote $E$ by $\{\emph{\textbf{k}}_{1}, \ldots, \emph{\textbf{k}}_{\#E}\}$. It follows from
$\{\varphi(\cdot-\emph{\textbf{k}}): \emph{\textbf{k}}\in \mathbb{Z}^{2}\}$ being linearly independent,
$\hbox{supp}(\varphi)\subseteq[N_{1}, M_{1}]\times
[N_{2}, M_{2}]$ and   $\hbox{supp}(f)\subseteq [a_{1}, b_{1}]\times [a_{2}, b_{2}]$
that there exists uniquely a finite  sequence  $\{c_{\emph{\textbf{k}}_{1}}, \ldots,
c_{\emph{\textbf{k}}_{\#E}}\}\subseteq \mathbb{C}$ such that
\begin{align}\label{XK} f=\sum^{\#E}_{l=1}c_{\emph{\textbf{k}}_{l}}\varphi(\cdot-\emph{\textbf{k}}_{l}).\end{align}
Now by \eqref{XK} and  \eqref{890756}, we have
\begin{align}\label{YK} \mathcal{R}_{\emph{\textbf{p}}}f=\sum^{\#E}_{l=1}c_{\emph{\textbf{k}}_{l}}\mathcal{R}_{\emph{\textbf{p}}}\varphi(\cdot-\emph{\textbf{p}}\textbf{\emph{\textbf{k}}}_{l}).\end{align}

$(\Longleftarrow)$:
Since  $\{\mathcal{R}_{\emph{\textbf{p}}}\varphi(\cdot-\emph{\textbf{p}}\emph{\textbf{k}}_{l}): l=1, \ldots, \#E\}$
is linearly independent in $L^{2}(\mathbb{R})$, then   $\{c_{\emph{\textbf{k}}_{l}}: l=1, \ldots, \#E\}$ can be determined
uniquely  by $\mathcal{R}_{\emph{\textbf{p}}}f$.
%
%
Note that $\{\varphi(\cdot-\emph{\textbf{k}})\}_{\emph{\textbf{k}}\in \mathbb{Z}^{2}}$ is linearly independent.
Then  with the sequence $\{c_{\emph{\textbf{k}}_{l}}: l=1, \ldots,  \#E\}$ at hand,  $f=\sum^{\#E}_{l=1}c_{\emph{\textbf{k}}_{l}}\varphi(\cdot-\emph{\textbf{k}}_{l})$ can be determined uniquely.
%

$(\Longrightarrow$):
If $\{\mathcal{R}_{\emph{\textbf{p}}}\varphi(\cdot-\emph{\textbf{p}}\textbf{\emph{\textbf{k}}}_{l}): l=1,
\ldots, \#E\}$
is linearly dependent, then
 there exists a nonzero sequence  $\{\widehat{c}_{\emph{\textbf{k}}_{l}}: l=1, \ldots, \#E\}$
 such that $\|\sum^{\#E}_{l=1}\widehat{c}_{\emph{\textbf{k}}_{l}}\mathcal{R}_{\emph{\textbf{p}}}\varphi(\cdot-\emph{\textbf{p}}\emph{\textbf{k}}_{l})\|_{L^2(\mathbb{R})}=0$.
 Recall that $\{\varphi(\cdot-\emph{\textbf{k}}): \emph{\textbf{k}}\in \mathbb{Z}^{2}\}$ is linearly independent.
 Then $\tilde{f}:=\sum^{\#E}_{l=1}\widehat{c}_{\emph{\textbf{k}}_{l}}\varphi(\cdot-\emph{\textbf{k}}_{l})\not\equiv0$
 but $\mathcal{R}_{\emph{\textbf{p}}}\tilde{f}\equiv0$.
 Now  $\tilde{f}$ is not distinguishable from $ g\equiv0\in V(\varphi, \mathbb{Z}^{2})$
since their   Radon transforms (w.r.t the direction vector $\emph{\textbf{p}}$) are both zero. This leads to a contradiction.
\end{proof}

\begin{rem}(1)
The sampling problem  is not considered  in Theorem \ref{main1}. Therefore $\mathcal{R}_{\emph{\textbf{p}}}\varphi$
is not required to be continuous therein. (2)
If the set  $\{\emph{\textbf{p}}\emph{\textbf{k}}: \emph{\textbf{k}}\in E\}$ is not  contained  in $\mathbb{Z}$, then it follows from \eqref{YK} that $\mathcal{R}_{\emph{\textbf{p}}}f$   sits in the quasi-SIS (QSIS) generated by $\mathcal{R}_{\emph{\textbf{p}}}\varphi$.
As addressed  in section \ref{Rieszbound}, the recovery   problem in QSIS is absolutely not the  trivial
generalization of that in SIS.
\end{rem}

The following subsection states that the SA Radon-based determination problem  in Theorem \ref{main1}
is absolutely  not trivial.

%
%

\subsection{A nontrivial problem: what  pair   $(\varphi, \emph{\textbf{p}})$ ensures    the system $\{\mathcal{R}_{\emph{\textbf{p}}}\varphi(\cdot-\emph{\textbf{p}}\emph{\textbf{k}}): \emph{\textbf{k}}\in E\}$ being  linearly independent}\label{bupingfan}
Note that  in Theorem \ref{main1}   the system $\{\mathcal{R}_{\emph{\textbf{p}}}\varphi(\cdot-\emph{\textbf{p}}\emph{\textbf{k}}): \emph{\textbf{k}}\in E\}$ is required to be linearly independent in $L^2(\mathbb{R}).$
Our purpose of  this subsection is to explain   that such a requirement  is absolutely not trivial.
The following lemma  is necessary for our discussion. It is derived from \cite[Lemma 6.7]{Wendlan}.

\begin{lem}\label{lemma1234}
Suppose that $\emph{\textbf{x}}_{k}\in \mathbb{R}^{d}, k=1, \ldots, N$ are pairwise distinct. Then the set  $\{e^{\texttt{i}\emph{\textbf{x}}_{k}\cdot\xi}\}^{N}_{k=1}$
is linearly independent on any interval $I\subseteq \mathbb{R}^{d}, $  namely,  for any   vector
$(\alpha_{1}, \ldots, \alpha_{N})\in \mathbb{C}^{N}$ if $\sum^{N}_{k=1}\alpha_{k}e^{\texttt{i}\emph{\textbf{x}}_{k}\cdot\xi}\equiv0$
then $(\alpha_{1}, \ldots, \alpha_{N})=\textbf{0}$.
\end{lem}

In what follows, we establish the  equivalent characterizations for  the linear independence of
$\{\mathcal{R}_{\emph{\textbf{p}}}\varphi(\cdot-\emph{\textbf{p}}\emph{\textbf{k}}): \emph{\textbf{k}}\in E\}$.

\begin{prop}\label{Remark3.2}
Let the  compactly supported  $\varphi$ and $E=\{\emph{\textbf{k}}_{1}, \ldots, \emph{\textbf{k}}_{\#E}\}\subseteq \mathbb{Z}^{2}$ be as in Theorem \ref{main1}. Then the following statements  are equivalent:

(1) The system $\{\mathcal{R}_{\emph{\textbf{p}}}\varphi(\cdot-\emph{\textbf{p}}\emph{\textbf{k}}_{j}):
j=1, \ldots, \#E\}$ is linearly independent in $L^2(\mathbb{R})$.

(2) For any vector $\textbf{0}\neq(c_{1}, \ldots, c_{\#E})^{T}\in \mathbb{C}^{\#E}$
it holds that
\begin{align}\label{98FG}
\int_{\mathbb{R}}|\sum^{\#E}_{j=1}c_je^{-\texttt{i}\emph{\textbf{p}}\emph{\textbf{k}}_{j}\xi}|^{2}|\widehat{\varphi}(\emph{\textbf{p}}^{T}\xi)|^{2}d\xi>0.
\end{align}

(3) $\widehat{\varphi}(\emph{\textbf{p}}^{T}\cdot)\not\equiv0,$ and if $\#E>1$ then  for any $j\neq n\in \{1, \ldots, \#E\}$ we have $\emph{\textbf{p}}\emph{\textbf{k}}_{j}\neq
\emph{\textbf{p}}\emph{\textbf{k}}_{n}$.
\end{prop}
\begin{proof}
By \eqref{40567} we can check that the Fourier transform of $\sum^{\#E}_{j=1}c_{j}\mathcal{R}_{\emph{\textbf{p}}}\varphi(\cdot-\emph{\textbf{p}}\textbf{\emph{\textbf{k}}}_{j})$
is $\sum^{\#E}_{j=1}c_je^{-\texttt{i}\emph{\textbf{p}}\emph{\textbf{k}}_{j}\xi}\widehat{\varphi}(\emph{\textbf{p}}^{T}\xi)$. From this we have     $(1)\Longleftrightarrow(2)$. If  $\widehat{\varphi}(\emph{\textbf{p}}\cdot)\equiv0$ then the integral in \eqref{98FG} is zero.
On the other hand, if $\#E>1$ and  $\emph{\textbf{p}}\emph{\textbf{k}}_{i_{1}}=\emph{\textbf{p}}\emph{\textbf{k}}_{i_{2}}$
for some  $i_{1}, i_{2}\in \{1, 2, \ldots, \#E\}$
 then
the integral is zero when  choosing $0\neq c_{i_{1}}=-c_{i_{2}}$ and $c_{j}=0$ for $j\neq i_{1}, i_{2}$. Then  $(2)\Longrightarrow(3)$.
Next we prove that $(3)\Longrightarrow(2)$. Actually, since $\varphi$ is compactly supported,
it follows from  Lemma \ref{pouy} that  $\mathcal{R}_{\emph{\textbf{p}}}\varphi$ is also  compactly supported.
Then $0\not\equiv\widehat{\mathcal{R}_{\emph{\textbf{p}}}\varphi}=\widehat{\varphi}(\emph{\textbf{p}}^{T}\cdot)\in C^{\infty}(\mathbb{R})$, and consequently  there exists an interval denoted by $[\xi_{0}-\delta_{0}, \xi_{0}+\delta_{0}]$
such that for any $\xi\in [\xi_{0}-\delta_{0}, \xi_{0}+\delta_{0}]$ we have $|\widehat{\mathcal{R}_{\emph{\textbf{p}}}\varphi}(\xi)|>0.$
Additionally, it follows from Lemma \ref{lemma1234} that $\{e^{-\texttt{i}\emph{\textbf{p}}\emph{\textbf{k}}_{l}}: l=1, \ldots,
\#E\}$ is linearly independent on $[\xi_{0}-\delta_{0}, \xi_{0}+\delta_{0}]$. Then
\begin{align}\label{92348FG}
\int_{\mathbb{R}}|\sum^{\#E}_{j=1}c_je^{-\texttt{i}\emph{\textbf{p}}\emph{\textbf{k}}_{j}\xi}|^{2}|\widehat{\varphi}(\emph{\textbf{p}}^{T}\xi)|^{2}d\xi\geq
\int^{\xi_{0}+\delta_{0}}_{\xi_{0}-\delta_{0}}|\sum^{\#E}_{j=1}c_je^{-\texttt{i}\emph{\textbf{p}}\emph{\textbf{k}}_{j}\xi}|^{2}|\widehat{\varphi}(\emph{\textbf{p}}^{T}\xi)|^{2}d\xi>0.
\end{align}
Consequently, \eqref{98FG} holds. This completes the proof.
\end{proof}

The following is a counterexample such that the condition in Proposition \ref{Remark3.2} is not satisfied.
Therefore, the problem of the linear independence of $\{\mathcal{R}_{\emph{\textbf{p}}}\varphi(\cdot-\emph{\textbf{p}}\emph{\textbf{k}}): \emph{\textbf{k}}\in E\}$ is not trivial.

\begin{exam}\label{examplenew}
The generator  $\varphi$ is defined such that
\begin{align}\label{nontrivalexample}\widehat{\varphi}(\xi_1, \xi_2)=\sin(\xi_1-\xi_2)\widehat{g}(\xi_1,\xi_2), \end{align}where $0\not\equiv g\in L^2(\mathbb{R}^{2})$
is compactly supported. One can check that  $\varphi(x_1,x_2)=\frac{1}{2\texttt{i}}g(x_1+1, x_2-1)-\frac{1}{2\texttt{i}}g(x_1-1, x_2+1)$
and is compactly supported as well.  Clearly,  $\widehat{\varphi}(\emph{\textbf{p}}^{T}\cdot)\equiv0$
if choosing   $\emph{\textbf{p}}=(\frac{\sqrt{2}}{2}, \frac{\sqrt{2}}{2})$.
\end{exam}

\textbf{Analysis with the help of Paley-Wiener theorem}.
From the perspective of zero distribution, $\widehat{\varphi}$ in Example \ref{examplenew} has zeros along the line $\xi_{1}-\xi_{2}=0$ on $\mathbb{R}^{2}=\{(\xi_{1}, \xi_{2})^{T}: \xi_{1}, \xi_{2}\in \mathbb{R}\}$.
This implies that $\widehat{\varphi}$ has non-isolated zeros on $\mathbb{R}^{2}$. For better understanding this issue,
in what follows we explain it from the perspective of zero distribution of entire functions.
The classical  Paley-Wiener theorem (c.f. \cite{Stein}) states that a function  $g\in L^{2}(\mathbb{R}^{d})$
is the Fourier transform of a square integrable function with compact support if and only if
it is the boundary value on $\mathbb{R}^{d}$ of an entire function on $\mathbb{C}^{d}$ of   exponential type.
Now for the compactly supported generator $\varphi\in L^2(\mathbb{R}^d)$, by the Paley-Wiener theorem we conclude that
its Fourier transform $\widehat{\varphi}$ is the boundary value on $\mathbb{R}^d$ of an entire function on $\mathbb{C}^{d}$.
It is well-known that for $d\geq2$ an entire function on $\mathbb{C}^{d}$ may have non-isolated zeros (c.f. \cite{duofubian}).
Therefore, it is no wonder that there exists a pair $(\varphi, \emph{\textbf{p}})$ such that
$\widehat{\varphi}(\emph{\textbf{p}}^{T}\xi)=0$ for any $\xi\in \mathbb{R}$. Correspondingly,  the system
$ \{\mathcal{R}_{\emph{\textbf{p}}}\varphi(\cdot-\emph{\textbf{p}}\emph{\textbf{k}}_{j}):
j=1, \ldots, \#E\}$ in Proposition \ref{Remark3.2}  is linearly dependent.

%
%
%
%
%
%

\section{SA-Radon samples  based  reconstruction   for  compactly supported   functions in    SIS}\label{boxspline}

This section concerns on  the SACT sampling  problem  \eqref{SPP}  for   compactly supported  functions  in the SIS    generated by  a   compactly supported generator $\varphi$. The main results will be  organized  in Theorems \ref{yibanshengcyuan}, \ref{xianshibiaoshi} and
\ref{yibanshengcyuanYY}. For the  better readability,  we quickly sketch the structure of this section.
A necessary and sufficient condition on $(\varphi, \emph{\textbf{p}})$ and the sampling set $X\subseteq \mathbb{R}$
will be  established in Theorem \ref{theorem123},
such that a compactly supported function  $f\in V(\varphi, \mathbb{Z}^2)$ can be determined uniquely by its SA Radon samples at $X.$
Based on Theorem \ref{theorem123},
our two main results are organized in Theorems \ref{yibanshengcyuan}, \ref{xianshibiaoshi}, and
Theorem \ref{yibanshengcyuanYY} and Proposition   \ref{hbvcxz}.
Theorems \ref{yibanshengcyuan} and  \ref{xianshibiaoshi} hold for the nonvanishing case ($\widehat{\varphi}(\textbf{0})\neq0$)
while Theorem  \ref{yibanshengcyuanYY} and  Proposition  \ref{hbvcxz} hold for the vanishing case ($\widehat{\varphi}(\textbf{0})=0$).



\subsection{ A sufficient and necessary    condition  on the pair $(\varphi, \emph{\textbf{p}})$ and the sampling set $X$ such that the SACT sampling \eqref{SPP} can be achieved.}\label{neirong1}
As previously, any  $\emph{\textbf{x}}\in \mathbb{R}^{2}$
  is considered as a column vector while the direction vector  $\emph{\textbf{p}}$ is a row vector.

\begin{theo}\label{theorem123}
Suppose that   $\varphi\in  L^{2}(\mathbb{R}^{2})$   such that    $\hbox{supp}(\varphi)\subseteq[N_{1}, M_{1}]\times
[N_{2}, M_{2}]$ and $\{\varphi(\cdot-\emph{\textbf{k}}): \emph{\textbf{k}}\in \mathbb{Z}^{2}\}$ is linearly independent, and $\emph{\textbf{p}}=(\cos\theta, \sin\theta)$ is a direction vector
such that $\mathcal{R}_{\emph{\textbf{p}}}\varphi$ is continuous.
Moreover,
$f\in V(\varphi, \mathbb{Z}^{2})$ is an arbitrary  target function  such that
$ \hbox{supp}(f)\subseteq [a_{1}, b_{1}]\times [a_{2}, b_{2}]$.
Let   $ E=\big\{\big[\lceil a_{1}-M_{1}\rceil, \lfloor b_{1}-N_{1}\rfloor\big]\times
\big[\lceil a_{2}-M_{2}\rceil, \lfloor b_{2}-N_{2}\rfloor\big]\big\}\cap \mathbb{Z}^{2}$
and denote it by
$\{\emph{\textbf{k}}_{1}, \ldots, \emph{\textbf{k}}_{\#E}\}.$
Then $f$ can be determined uniquely by its
SA Radon (w.r.t  $\emph{\textbf{p}}$) samples at $X=\{x_{1}, \ldots, x_{\#E}\}\subseteq \mathbb{R}$ if  and only if  the $\#E\times \#E$ matrix
\begin{align}\label{qixiang1234}
A_{\varphi,\emph{\textbf{p}}, X}:=\left(\begin{array}{cccccccccc}
\mathcal{R}_{\emph{\textbf{p}}}\varphi(x_{1}-\emph{\textbf{p}}\emph{\textbf{k}}_{1})&\mathcal{R}_{\emph{\textbf{p}}}\varphi(x_{1}-\emph{\textbf{p}}\emph{\textbf{k}}_{2})&\cdots&\mathcal{R}_{\emph{\textbf{p}}}\varphi(x_{1}-
\emph{\textbf{p}}\emph{\textbf{k}}_{\#E})\\
\mathcal{R}_{\emph{\textbf{p}}}\varphi(x_{2}-\emph{\textbf{p}}\emph{\textbf{k}}_{1})&\mathcal{R}_{\emph{\textbf{p}}}\varphi(x_{2}-\emph{\textbf{p}}\emph{\textbf{k}}_{2})&\cdots&\mathcal{R}_{\emph{\textbf{p}}}\varphi(x_{2}-
\emph{\textbf{p}}\emph{\textbf{k}}_{\#E})\\
\vdots&\vdots&\ddots&\vdots\\
\mathcal{R}_{\emph{\textbf{p}}}\varphi(x_{\#E}-
\emph{\textbf{p}}\emph{\textbf{k}}_{1})&\mathcal{R}_{\emph{\textbf{p}}}\varphi(x_{\#E}-
\emph{\textbf{p}}\emph{\textbf{k}}_{2})&\cdots&\mathcal{R}_{\emph{\textbf{p}}}\varphi(x_{\#E}-\emph{\textbf{p}}\emph{\textbf{k}}_{\#E})
\end{array}\right)
\end{align}
is invertible.
\end{theo}
\begin{proof}
($\Longleftarrow$) We first prove that
if $A_{\varphi,\emph{\textbf{p}}, X}$ is invertible then
$\{\mathcal{R}_{\textbf{\emph{p}}}\varphi(\cdot-\emph{\textbf{p}}\emph{\textbf{k}}_{n}): n=1, \ldots,  \#E\}$ is linearly independent
in $L^2(\mathbb{R})$. Otherwise,   there exists a nonzero vector
$(\widehat{d}_{1}, \ldots, \widehat{d}_{\#E})^{T}\in \mathbb{C}^{\#E}$ such that \begin{align}\label{gnewvbc}\|\sum^{\#E}_{n=1}\widehat{d}_{n}\mathcal{R}_{\textbf{\emph{p}}}\varphi(\cdot-\emph{\textbf{p}}\emph{\textbf{k}}_{n})\|^{2}_{L^2(\mathbb{R})}=
\int_{\mathbb{R}}|\sum^{\#E}_{n=1}\widehat{d}_{n}\mathcal{R}_{\textbf{\emph{p}}}\varphi(x-\emph{\textbf{p}}\emph{\textbf{k}}_{n})|^{2}dx=0.\end{align}
It follows from \eqref{gnewvbc} and the continuity of $\mathcal{R}_{\textbf{\emph{p}}}\varphi$ that
for any $x_{l}\in X$ we have $\sum^{\#E}_{n=1}\widehat{d}_{n}\mathcal{R}_{\emph{\textbf{p}}}\varphi(x_{l}-\emph{\textbf{p}}\emph{\textbf{k}}_{n})=0,$
which implies that the matrix   $A_{\varphi,\emph{\textbf{p}}, X}$ is singular.
This is a contradiction.
Next we    prove that
$\mathcal{R}_{\emph{\textbf{p}}}f$ can be determined by its samples at $X$ if
$A_{\varphi,\emph{\textbf{p}}, X}$ is invertible.

As in \eqref{XK} and \eqref{YK}, there exists uniquely $ (c_{\emph{\textbf{k}}_{1}},
\ldots, c_{\emph{\textbf{k}}_{\#E}})^{T}\in \mathbb{C}^{\#E}$ such that
\begin{align}\label{XKYYZ}\begin{array}{llll} \displaystyle f=\sum^{\#E}_{n=1}c_{\emph{\textbf{k}}_{n}}\varphi(\cdot-\emph{\textbf{k}}_{n})
\end{array} \end{align}
and consequently,
\begin{align}\label{Y00K}\begin{array}{llll} \displaystyle\mathcal{R}_{\emph{\textbf{p}}}f=\sum^{\#E}_{n=1}c_{\emph{\textbf{k}}_{n}}\mathcal{R}_{\emph{\textbf{p}}}\varphi(\cdot-\emph{\textbf{p}}\textbf{\emph{\textbf{k}}}_{n}).
\end{array} \end{align}
Now it follows from \eqref{Y00K} that
\begin{align} \label{ui981207}A_{\varphi,\emph{\textbf{p}}, X}(c_{\emph{\textbf{k}}_{1}},
\ldots, c_{\emph{\textbf{k}}_{\#E}})^{T}=(\mathcal{R}_{\emph{\textbf{p}}}f(x_{1}), \ldots, \mathcal{R}_{\emph{\textbf{p}}}f(x_{\#E}))^{T}.\end{align}
Since $A_{\varphi,\emph{\textbf{p}}, X}$ is invertible then
$(c_{\emph{\textbf{k}}_{1}},
\ldots, c_{\emph{\textbf{k}}_{\#E}})^{T}$
can be determined uniquely by the SA Radon samples $\mathcal{R}_{\emph{\textbf{p}}}f(x_{1}),
\ldots, \mathcal{R}_{\emph{\textbf{p}}}f(x_{\#E})$.
Since $\{\mathcal{R}_{\textbf{\emph{p}}}\varphi(\cdot-\emph{\textbf{p}}\emph{\textbf{k}}_{n}): n=1, \ldots,  \#E\}$
is   linearly independent, $\mathcal{R}_{\emph{\textbf{p}}}f$ represented  via  \eqref{Y00K} can be  determined from
the vector $(c_{\emph{\textbf{k}}_{1}},
\ldots, c_{\emph{\textbf{k}}_{\#E}})^{T}$.
Now  by Theorem \ref{main1},
 $f=\sum^{\#E}_{n=1}c_{\emph{\textbf{k}}_{n}}\varphi(\cdot-\emph{\textbf{k}}_{n})$ can be determined uniquely.

 $(\Longrightarrow)$ If $A_{\varphi, \emph{\textbf{p}}, X}$ is not invertible then
$ (c_{\emph{\textbf{k}}_{1}},
\ldots, c_{\emph{\textbf{k}}_{\#E}})^{T}$ can not be determined uniquely by
\eqref{ui981207}. Recall again that $\{\varphi(\cdot-\emph{\textbf{k}}): \emph{\textbf{k}}\in E\}$
is linearly independent, then $f$ in \eqref{XKYYZ} can not be determined uniquely.
\end{proof}

\begin{rem}\label{alternative}
For the sampling problem in the SIS $V(\varphi, \mathbb{Z}^2)$, it is required that  $\varphi$ is    continuous  (c.f.
Aldroubi and   Gr\"{o}chenig \cite{Fienup289}).
Therefore, if $\varphi$ is discontinuous then the sampling in   $V(\varphi, \mathbb{Z}^2)$
is not well-defined.
On the other hand, it follows from Remark \ref{kenenglianxu} that   even though  $\varphi$ is   discontinuous,
the Radon transform $\mathcal{R}_{\emph{\textbf{p}}}\varphi$ may be continuous.
From this perspective,  when $\varphi$ is discontinuous Theorem \ref{theorem123} may  provide  an alternative sampling-based recovery for compactly supported functions in
$V(\varphi, \mathbb{Z}^2)$.
\end{rem}

\subsection{Direction vector set and null set}
The concepts of direction vector set and null set
will be  necessary  for  SACT sampling.
\begin{defi}\label{crictial}
(1) Suppose that $\mathcal{S}\subseteq \mathbb{R}^{2}$ such that $\mathcal{S}\backslash\{\textbf{0}\}$ is not empty. Define its direction vector set as \begin{align}\label{dvdingyi}\hbox{dv}_{\mathcal{S}}=
\{(\cos\theta, \sin\theta): \hbox{all} \ \textbf{0}\neq\emph{\textbf{x}}=\|\emph{\textbf{x}}\|_{2}(\cos\theta, \sin\theta)^{T}\in \mathcal{S}\}.
\end{align}
The direction vector sets of the  empty set $\emptyset$ and $\{\textbf{0}\}$ are both  simply defined as $\emptyset$.

(2) For $\mathcal{S}\subseteq \mathbb{R}^{2}$ such that  $\mathcal{S}\backslash\{\textbf{0}\}$ is not empty,
its   null set   $\mathcal{N}_{\mathcal{S}}$
is defined as \begin{align}\label{nullset}\{\textbf{0}\neq\emph{\textbf{y}}\in \mathbb{R}^{2}: \hbox{there exists} \ \textbf{0}\neq\emph{\textbf{x}}\in \mathcal{S} \
\hbox{such that} \ \emph{\textbf{x}}^{T}\emph{\textbf{y}}=0\}.\end{align}
The null sets of $\emptyset$ and $\{\textbf{0}\}$ are both simply defined as $\emptyset$.
Correspondingly, if $\mathcal{S}\backslash\{\textbf{0}\}$ is not empty then  the direction vector set $\hbox{dv}_{\mathcal{N}_{\mathcal{S}}}$
is defined via \eqref{dvdingyi}.
\end{defi}

\begin{rem}\label{guanyucedu} (1)
For $\emph{\textbf{x}}_{0}\in \mathbb{R}^2$ and its  open  disc  \begin{align}
\label{disck}\mathring{\mathcal{D}}(\emph{\textbf{x}}_{0}, \delta):=\{\emph{\textbf{x}}\in \mathbb{R}^{2}: \|\emph{\textbf{x}}-\emph{\textbf{x}}_{0}\|_{2}<\delta\},\end{align}  if  $\textbf{0}\in \mathring{\mathcal{D}}(\emph{\textbf{x}}_{0}, \delta)$ then $\hbox{dv}_{\mathring{\mathcal{D}}(\emph{\textbf{x}}_{0}, \delta)}$ is the unit circle $\{(\cos\theta, \sin \theta): \theta\in [0, 2\pi)\}$.
(2)
Suppose that $\mathcal{S}\subseteq \mathbb{R}^{2}$ is finite
such that $\mathcal{S}\backslash\{\textbf{0}\}$ is not empty.
The  null set $\mathcal{N}_{\mathcal{S}}$ of $\mathcal{S}$ is defined via \eqref{nullset}.
  Then its  cardinality  $\#\hbox{dv}_{\mathcal{N}_{\mathcal{S}}}<\infty$.
\end{rem}
\begin{proof}
Item (1) is obvious. We just need to prove item (2).
Denote $\mathcal{S}\backslash\{\textbf{0}\}$ by $\{\emph{\textbf{x}}_{1}, \ldots, \emph{\textbf{x}}_{L}\}$.
For any $\textbf{0}\neq\emph{\textbf{x}}_{l}=(\emph{\textbf{x}}_{l,1}, \emph{\textbf{x}}_{l, 2})^{T}\in \mathcal{S}$,
suppose that $\textbf{0}\neq\emph{\textbf{y}}=\|\emph{\textbf{y}}\|_{2}(\cos\theta_{\emph{\textbf{y}}},$
$ \sin\theta_{\emph{\textbf{y}}})^{T}$
such that $\emph{\textbf{x}}^{T}_{l}\emph{\textbf{y}}=0.$ Without loss of generality, let $\emph{\textbf{x}}_{l,2}\neq0$.
Then $\tan\theta_{\emph{\textbf{y}}}=-\frac{\emph{\textbf{x}}_{l,1}}{\emph{\textbf{x}}_{l,2}}$.
By $\mathcal{S}$ being  finite, the proof can be completed.
\end{proof}

The following direction vector set is related to a function.

\begin{defi}\label{definitin1234}
Suppose that $0\not\equiv g: \mathbb{R}^2\rightarrow \mathbb{C}$ is continuous. For   $\emph{\textbf{x}}_{0}\in
\mathbb{R}^{2}$ such that $g(\emph{\textbf{x}}_{0})\neq0$, let $\delta^{g}_{\emph{\textbf{x}}_{0}, \max}>0$
be the maximum value in $(0, \infty]$ such that for any $\emph{\textbf{x}}\in \mathring{\mathcal{D}}(\emph{\textbf{x}}_0,
\delta^{g}_{\emph{\textbf{x}}_{0}, \max})$ we have $g(\emph{\textbf{x}})\neq0,$ where $ \mathring{\mathcal{D}}(\emph{\textbf{x}}_0,
\delta^{g}_{\emph{\textbf{x}}_{0}, \max})$ is defined via \eqref{disck}.
Following Definition \ref{crictial} \eqref{dvdingyi},
the set of direction vectors $\hbox{dv}_{\mathring{\mathcal{D}}(\emph{\textbf{x}}_{0}, \delta^{g}_{\emph{\textbf{x}}_{0},\max})}$ of $\mathring{\mathcal{D}}(\emph{\textbf{x}}_{0}, \delta^{g}_{\emph{\textbf{x}}_{0},\max})$
is defined as
\begin{align}\label{416}\{(\cos\theta, \sin\theta): \textbf{0}\neq\emph{\textbf{x}}=\|\emph{\textbf{x}}\|_{2}(\cos\theta, \sin\theta)^{T}\in \mathring{\mathcal{D}}(\emph{\textbf{x}}_{0}, \delta^{g}_{\emph{\textbf{x}}_{0},\max})\}.\end{align}
\end{defi}

\begin{defi}\label{flagvector}
Suppose that $0\not\equiv\varphi\in L^{2}(\mathbb{R}^2)$ is compactly supported and vanishing
(i.e. $\widehat{\varphi}(\textbf{0})=0$).
Denote the nonzero set of $\widehat{\varphi}$ by $\mathcal{G}_{\widehat{\varphi}}$ such that $\widehat{\varphi}(\emph{\textbf{x}})\neq0$
for any $\emph{\textbf{x}}\in \mathcal{G}_{\widehat{\varphi}}$. Define
\begin{align}\label{DVV123}\hbox{DV}_{\widehat{\varphi}}=\bigcup_{\emph{\textbf{x}}\in \mathcal{G}_{\widehat{\varphi}}}\hbox{dv}_{\mathring{\mathcal{D}}(\emph{\textbf{x}},
\delta^{\widehat{\varphi}}_{\emph{\textbf{x}}, \max})},\end{align}
where $\hbox{dv}_{\mathring{\mathcal{D}}(\emph{\textbf{x}},
\delta^{\widehat{\varphi}}_{\emph{\textbf{x}}, \max})}$ is defined via  Definition \ref{definitin1234}.
Correspondingly, the angle set of $\hbox{DV}_{\widehat{\varphi}}$ is defined  as
\begin{align}\label{8tyr}\arg_{\hbox{DV}_{\widehat{\varphi}}} =\{\theta\in [0, 2\pi): (\cos\theta, \sin\theta)\in \hbox{DV}_{\widehat{\varphi}}\}.\end{align}
\end{defi}

\begin{prop}\label{lebesguecedu}
Let $\varphi$ and $\arg_{\hbox{DV}_{\widehat{\varphi}}} $ be as in Definition \ref{flagvector}.  Then
(1) the Lebesgue measure  $\mu(\arg_{\hbox{DV}_{\widehat{\varphi}}})$ of $\arg_{\hbox{DV}_{\widehat{\varphi}}} $ on
$\mathbb{R}$ is positive;  (2) $\widehat{\mathcal{R}_{\emph{\textbf{p}}}\varphi}=\widehat{\varphi}(\emph{\textbf{p}}^{T}\cdot)\not\equiv0$ for  $\emph{\textbf{p}}=(\cos\theta, \sin\theta)$ with any $\theta\in \arg_{\hbox{DV}_{\widehat{\varphi}}}$.
\end{prop}
\begin{proof}
We first prove item (1).
Since $0\not\equiv\varphi \in L^2(\mathbb{R}^2)$ is compactly supported,   $0\not\equiv\widehat{\varphi}\in C^{\infty}(\mathbb{R}^2).$
Then the nonzero set  $\mathcal{G}_{\widehat{\varphi}}$ of $\widehat{\varphi}$ is not empty.  Choose any $\emph{\textbf{x}}\in \mathcal{G}_{\widehat{\varphi}}$
and consider $\hbox{dv}_{\mathring{\mathcal{D}}(\emph{\textbf{x}},
\delta^{\widehat{\varphi}}_{\emph{\textbf{x}}, \max})}.$ As in \eqref{8tyr}, define the angle set of
$\hbox{dv}_{\mathring{\mathcal{D}}(\emph{\textbf{x}},
\delta^{\widehat{\varphi}}_{\emph{\textbf{x}}, \max})}$ as $\arg_{\hbox{dv}_{\mathring{\mathcal{D}}(\emph{\textbf{x}},
\delta^{\widehat{\varphi}}_{\emph{\textbf{x}}, \max})}}=\{\theta\in [0, 2\pi): (\cos\theta, \sin\theta)\in
\hbox{dv}_{\mathring{\mathcal{D}}(\emph{\textbf{x}},
\delta^{\widehat{\varphi}}_{\emph{\textbf{x}}, \max})}\}.$
Since $\delta^{\widehat{\varphi}}_{\emph{\textbf{x}}, \max}>0$,
 the Lebesgue measure $\mu(\arg_{\hbox{dv}_{\mathring{\mathcal{D}}(\emph{\textbf{x}},
\delta^{\widehat{\varphi}}_{\emph{\textbf{x}}, \max})}})>0.$ Therefore, $\mu(\arg_{\hbox{DV}_{\widehat{\varphi}}})$
$>0.$

Next we prove item (2). For any $\theta\in \arg_{\hbox{DV}_{\widehat{\varphi}}}$, by \eqref{8tyr} the  corresponding direction vector
 $\emph{\textbf{p}}=(\cos\theta, \sin\theta)$
 $\in \hbox{DV}_{\widehat{\varphi}}$. Now by \eqref{DVV123}
 there exists $\emph{\textbf{x}}\in \mathcal{G}_{\widehat{\varphi}}$
 such that $\widehat{\varphi}(\emph{\textbf{x}})\neq0$ and $\emph{\textbf{p}}\in
 \hbox{dv}_{\mathring{\mathcal{D}}(\emph{\textbf{x}},
\delta^{\widehat{\varphi}}_{\emph{\textbf{x}}, \max})}.$ By the definition of $ \mathring{\mathcal{D}}(\emph{\textbf{x}},
\delta^{\widehat{\varphi}}_{\emph{\textbf{x}}, \max})$ in Definition \ref{definitin1234}, there exists $\gamma>0$
such that $\widehat{\varphi}(\gamma \emph{\textbf{p}}^{T})\neq0$. Therefore,   $\widehat{\mathcal{R}_{\emph{\textbf{p}}}\varphi}(\gamma)=\widehat{\varphi}(\gamma \emph{\textbf{p}}^{T})\neq0$.
By Lemma \ref{pouy}, $\mathcal{R}_{\emph{\textbf{p}}}\varphi$ is compactly supported and consequently
  $\widehat{\mathcal{R}_{\emph{\textbf{p}}}\varphi}\in C^{\infty}(\mathbb{R})$.
Now by the continuity of $\widehat{\mathcal{R}_{\emph{\textbf{p}}}\varphi}$ one can prove that
$\widehat{\mathcal{R}_{\emph{\textbf{p}}}\varphi}\not\equiv0$.
This completes the proof.
\end{proof}

\subsection{The first main result: SACT sampling for compactly supported functions in a SIS generated by a non-vanishing generator $\varphi$}
\label{1stresult}
The following is the first main theorem in this section.

\begin{theo}\label{yibanshengcyuan}
Suppose that $\varphi\in L^{2}(\mathbb{R}^{2})$ is compactly supported   such that
$\hbox{supp}(\varphi) \subseteq[N_{1}, M_{1}]\times [N_{2}, M_{2}]$, $\{\varphi(\cdot-\emph{\textbf{k}}): \emph{\textbf{k}}\in \mathbb{Z}^{2}\}$
is linearly independent and

(i) the  Sobolev smoothness $\nu_{2}(\varphi)>1/2$,

(ii)  $\widehat{\varphi}(\textbf{0})=
\int_{\mathbb{R}^{2}}\varphi(\emph{\textbf{x}})d\emph{\textbf{x}}\neq0$ (non-vanishing property). \\
As previously, suppose that $f\in V(\varphi, \mathbb{Z}^2)$ is an arbitrary target function
such that $\hbox{supp}(f)\subseteq[a_{1}, b_{1}]\times [a_{2}, b_{2}]$. Correspondingly, define two sets
$$ E=\big\{\big[\lceil a_{1}-M_{1}\rceil, \lfloor b_{1}-N_{1}\rfloor\big]\times
\big[\lceil a_{2}-M_{2}\rceil, \lfloor b_{2}-N_{2}\rfloor\big]\big\}\cap \mathbb{Z}^{2}$$
and
\begin{align}E^{+}=\left\{
\begin{array}{cccc}
   \emptyset,&\#E=1,\\
  \{\emph{\textbf{x}}-\emph{\textbf{y}}:
\emph{\textbf{x}}\neq\emph{\textbf{y}}\in E\},&\#E>1.
\end{array}
\right.\end{align}
Then for any $\emph{\textbf{p}}\in \{(\cos\theta, \sin\theta): \theta\in [0, 2\pi)\}\setminus\hbox{dv}_{\mathcal{N}_{E^{+}}}$,
there exists a sampling   set $X_{\emph{\textbf{p}}}\subseteq \mathbb{R}$ having the  cardinality $\#X_{\emph{\textbf{p}}}=\#E$ such that
$f$ can be determined by its  SA Radon (w.r.t $\emph{\textbf{p}}$) samples at $X_{\emph{\textbf{p}}}$.
\end{theo}
\begin{proof}
Denote $E=\{\emph{\textbf{k}}_{1}, \ldots, \emph{\textbf{k}}_{\#E}\}$.
We first prove  $\{(\cos\theta, \sin\theta): \theta\in [0, 2\pi)\}\setminus\hbox{dv}_{\mathcal{N}_{E^{+}}}$
is not empty. It is sufficient to prove  that $\#\hbox{dv}_{\mathcal{N}_{E^{+}}}<\infty$.
If $\#E=1$ then $E^{+}=\emptyset$ and by Definition \ref{crictial} (1) we have
$\hbox{dv}_{\mathcal{N}_{E^{+}}}=\emptyset$ and $\#\hbox{dv}_{\mathcal{N}_{E^{+}}}=0$.
If $\#E>1$ then $\#E^{+}=\#E(\#E-1)<\infty$. By Proposition  \ref{guanyucedu} (2) we have $\#\hbox{dv}_{\mathcal{N}_{E^{+}}}<\infty$.

Since $\hbox{supp}(\varphi)\subseteq [N_{1}, M_{1}]\times
 [N_{2}, M_{2}]$ and $\hbox{supp}(f)\subseteq[a_{1}, b_{1}]\times [a_{2}, b_{2}]$,
 as in \eqref{XKYYZ}  we denote   $f=\sum^{\#E}_{l=1}c_{\emph{\textbf{k}}_{l}}\varphi(\cdot-\emph{\textbf{k}}_{l})$
 for $(c_{\emph{\textbf{k}}_{1}}, \ldots, c_{\emph{\textbf{k}}_{\#E}})\in \mathbb{C}^{\#E}$.
Consequently,  by \eqref{Y00K} we have \begin{align}\label{9xy0876}\mathcal{R}_{\emph{\textbf{p}}}f=\sum^{\#E}_{l=1}c_{\emph{\textbf{k}}_{l}}\mathcal{R}_{\emph{\textbf{p}}}\varphi(\cdot-\emph{\textbf{p}}\emph{\textbf{k}}_{l}).\end{align}
We first prove that for any $\emph{\textbf{p}}\in \{(\cos\theta, \sin\theta): \theta\in [0, 2\pi)\}\setminus\hbox{dv}_{\mathcal{N}_{E^{+}}}$, the system $\{\mathcal{R}_{\emph{\textbf{p}}}\varphi(\cdot-\emph{\textbf{p}}\emph{\textbf{k}}_{l}): l=1, \ldots, \#E\}$
is linearly independent. For the equivalence of the linear independence established in Proposition \ref{Remark3.2} for  the above system,
we just need to prove that Proposition \ref{Remark3.2} (3) is satisfied
any $\emph{\textbf{p}}\in \{(\cos\theta, \sin\theta): \theta\in [0, 2\pi)\}\setminus\hbox{dv}_{\mathcal{N}_{E^{+}}}$.
Clearly, $\widehat{\mathcal{R}_{\emph{\textbf{p}}}\varphi}(0)=\widehat{\varphi}(\textbf{0})\neq0$
for any $\emph{\textbf{p}}$. Then  \begin{align}\label{PLX}\widehat{\mathcal{R}_{\emph{\textbf{p}}}\varphi}=\widehat{\varphi}(\emph{\textbf{p}}^{T}\cdot)\not\equiv0.\end{align}
On the other hand, if $E^{+}=\emptyset$ then $\hbox{dv}_{\mathcal{N}_{E^{+}}}=\emptyset$. This combining \eqref{PLX} implies that    item (3)
of Proposition \ref{Remark3.2} is naturally  satisfied for any $\emph{\textbf{p}}\in \{(\cos\theta, \sin\theta):
\theta\in [0, 2\pi)\}$. If  $E^{+}\neq\emptyset$ then
it follows from the  definition of $\mathcal{N}_{E^{+}}$ in Definition \ref{crictial} (2)
that for any $\emph{\textbf{p}}\notin\hbox{dv}_{\mathcal{N}_{E^{+}}}$ we have
 $\emph{\textbf{p}}\emph{\textbf{k}}_{l}\neq \emph{\textbf{p}}\emph{\textbf{k}}_{n}$
for any $l\neq n\in \{1, \ldots, \#E\}$. That is, for  the case that $E^{+}\neq\emptyset$  item (3)
of Proposition \ref{Remark3.2} is also  satisfied.    Then it follows from Proposition \ref{Remark3.2} that  $\{\mathcal{R}_{\emph{\textbf{p}}}\varphi(\cdot-\emph{\textbf{p}}\emph{\textbf{k}}_{l}): l=1, \ldots, \#E\}$ is linearly independent.

By the above independence   there exist constants  $0<C_{1,\emph{\textbf{p}}}\leq C_{2,\emph{\textbf{p}}}<\infty$ such that
\begin{align}\label{ffttt}
C_{1,\emph{\textbf{p}}}\sum^{\#E}_{l=1}|d_{\emph{\textbf{k}}_{l}}|^{2}\leq\int_{\mathbb{R}}|\sum^{\#E}_{l=1}d_{\emph{\textbf{k}}_{l}}\mathcal{R}_{\emph{\textbf{p}}}\varphi(x-\emph{\textbf{p}}\emph{\textbf{k}}_{l})|^{2}dx\leq
C_{2,\emph{\textbf{p}}}\sum^{\#E}_{l=1}|d_{\emph{\textbf{k}}_{l}}|^{2}
\end{align}
for any $(d_{\emph{\textbf{k}}_{1}}, \ldots, d_{\#E})^{T}\in \mathbb{R}^{\#E}$.
On  the other hand, it follows from Proposition \ref{pouy} \eqref{qujian} that
$\hbox{supp}(\mathcal{R}_{\emph{\textbf{p}}}\varphi)\subseteq [-L_{\varphi}, L_{\varphi}]$,
where $L_{\varphi}=\sqrt{2}\max\{|N_{i}|, |M_{i}|: i=1, 2\}$. Denote $a_{\emph{\textbf{p}},1}=\min\{{\emph{\textbf{p}}}\emph{\textbf{k}}_{l}:
l=1, \ldots, \#E\}$ and $a_{\emph{\textbf{p}},2}=\max\{{\emph{\textbf{p}}}\emph{\textbf{k}}_{l}:
l=1, \ldots, \#E\}$. One can check that
$$\hbox{supp}\big(\sum^{\#E}_{l=1}d_{\emph{\textbf{k}}_{l}}\mathcal{R}_{\emph{\textbf{p}}}\varphi(\cdot-\emph{\textbf{p}}\emph{\textbf{k}}_{l})\big)\subseteq[
L_{\emph{\textbf{p}},1}, L_{\emph{\textbf{p}},2}],$$
 where $L_{\emph{\textbf{p}},1}=-L_{\varphi}+a_{\emph{\textbf{p}},1}$
and $L_{\emph{\textbf{p}},2}=L_{\varphi}+a_{\emph{\textbf{p}},2}$.
Then \eqref{ffttt} is equivalent to
\begin{align}\label{ffttt134}
C_{1,\emph{\textbf{p}}}\sum^{\#E}_{l=1}|d_{\emph{\textbf{k}}_{l}}|^{2}\leq\int^{L_{\emph{\textbf{p}},2}}_{L_{\emph{\textbf{p}},1}}|\sum^{\#E}_{l=1}d_{\emph{\textbf{k}}_{l}}\mathcal{R}_{\emph{\textbf{p}}}\varphi(x-\emph{\textbf{p}}\emph{\textbf{k}}_{l})|^{2}dx\leq
C_{2,\emph{\textbf{p}}}\sum^{\#E}_{l=1}|d_{\emph{\textbf{k}}_{l}}|^{2}.
\end{align}
The rest of the proof is to find  a sampling set $X_{\emph{\textbf{p}}}\subseteq \mathbb{R}$
with the cardinality $\#X_{\emph{\textbf{p}}}=\#E$ such that
$f$ can be determined by its  SA Radon (w.r.t $\emph{\textbf{p}}$) samples at $X_{\emph{\textbf{p}}}$.
Since $\nu_{2}(\varphi)>1/2$,
by Proposition  \ref{lianxuxing} (2) we have that $\mathcal{R}_{\emph{\textbf{p}}}\varphi$ is continuous.
Consequently, all  $\mathcal{R}_{\emph{\textbf{p}}}\varphi(\cdot-\emph{\textbf{p}}\emph{\textbf{k}}_{l}),
l=1, \ldots, \#E$
are uniformly continuous  on the interval $[L_{\emph{\textbf{p}},1}, L_{\emph{\textbf{p}},2}]$.
Then  there exists $\delta_{\emph{\textbf{p}}}\leq (L_{\emph{\textbf{p}},2}-L_{\emph{\textbf{p}},1})$ such that for
 any $l\in \{1, \ldots, \#E\}$ and
 any $x^{'}, x^{''}\in [L_{\emph{\textbf{p}},1}, L_{\emph{\textbf{p}},2}]$  satisfying $|x^{'}-x^{''}|<\delta_{\emph{\textbf{p}}}$ we have
\begin{align}\label{dkkds}|\mathcal{R}_{\emph{\textbf{p}}}\varphi(x^{'}-\emph{\textbf{p}}\emph{\textbf{k}}_{l})-\mathcal{R}_{\emph{\textbf{p}}}\varphi(x^{''}-\emph{\textbf{p}}\emph{\textbf{k}}_{l})|\leq\sqrt{\frac{C_{1,\emph{\textbf{p}}}}{3\#E(L_{\emph{\textbf{p}},2}-L_{\emph{\textbf{p}},1})}}.\end{align}
Now let $K_{\emph{\textbf{p}}}=\lceil\frac{L_{\emph{\textbf{p}},2}-L_{\emph{\textbf{p}},1}}{\delta_{\emph{\textbf{p}}}}\rceil$. Construct $$Y_{\emph{\textbf{p}}}=\{x_{k}=L_{\emph{\textbf{p}},1}+\frac{L_{\emph{\textbf{p}},2}-L_{\emph{\textbf{p}},1}}{K_{\emph{\textbf{p}}}}(k-1): k=1, \ldots, K_{\emph{\textbf{p}}}+1\}$$
such that \begin{align}\label{KPPP} |x_{k}-x_{j}|\leq\delta_{\emph{\textbf{p}}}\end{align} for any $x_{k}, x_{j}$.
Define an approximation to $\mathcal{R}_{\emph{\textbf{p}}}(\cdot-\emph{\textbf{p}}\emph{\textbf{k}}_{l})$ as $h_{l}(x)=\sum^{K_{\emph{\textbf{p}}}}_{k=1}\mathcal{R}_{\emph{\textbf{p}}}\varphi(x_{k}-\emph{\textbf{p}}\emph{\textbf{k}}_{l})\chi_{[x_{k},x_{k+1})}(x)$.
Then    one can check that
\begin{align}\label{cxznew}\begin{array}{lllllllll}
\displaystyle \int^{L_{\emph{\textbf{p}},2}}_{L_{\emph{\textbf{p}},1}}|\sum^{\#E}_{l=1}d_{\emph{\textbf{k}}_{l}}(\mathcal{R}_{\emph{\textbf{p}}}\varphi(x-\emph{\textbf{p}}\emph{\textbf{k}}_{l})-h_{l}(x))|^{2}dx\\
\displaystyle \leq\sum^{\#E}_{j=1}|d_{\emph{\textbf{k}}_{j}}|^{2}\int^{L_{\emph{\textbf{p}},2}}_{L_{\emph{\textbf{p}},1}}\sum^{\#E}_{l=1}|\mathcal{R}_{\emph{\textbf{p}}}\varphi(x-\emph{\textbf{p}}\emph{\textbf{k}}_{l})-h_{l}(x)|^{2}dx \ (\ref{cxznew} A)\\
\displaystyle =\sum^{\#E}_{j=1}|d_{\emph{\textbf{k}}_{j}}|^{2}\sum^{K_{\emph{\textbf{p}}}}_{n=1}\int^{x_{n+1}}_{x_{n}}\sum^{\#E}_{l=1}|\mathcal{R}_{\emph{\textbf{p}}}\varphi(x-\emph{\textbf{p}}\emph{\textbf{k}}_{l})-h_{l}(x)|^{2}dx
\end{array}
\end{align}
where (\ref{cxznew}A) is derived from the   Cauchy-Schwart inequality.
We continue to estimate (\ref{cxznew}) as follows,

\begin{align}\label{cxz}\begin{array}{lllllllll}
\displaystyle
\int^{L_{\emph{\textbf{p}},2}}_{L_{\emph{\textbf{p}},1}}|\sum^{\#E}_{l=1}d_{\emph{\textbf{k}}_{l}}(\mathcal{R}_{\emph{\textbf{p}}}\varphi(x-\emph{\textbf{p}}\emph{\textbf{k}}_{l})-h_{l}(x))|^{2}dx\\
\leq\displaystyle \sum^{\#E}_{j=1}|d_{\emph{\textbf{k}}_{j}}|^{2}\sum^{K_{\emph{\textbf{p}}}}_{n=1}\int^{x_{n+1}}_{x_{n}}\sum^{\#E}_{l=1}|\mathcal{R}_{\emph{\textbf{p}}}\varphi(x-\emph{\textbf{p}}\emph{\textbf{k}}_{l})-h_{l}(x)|^{2}dx\\
=\displaystyle \sum^{\#E}_{j=1}|d_{\emph{\textbf{k}}_{j}}|^{2}\sum^{K_{\emph{\textbf{p}}}}_{n=1}\int^{x_{n+1}}_{x_{n}}\sum^{\#E}_{l=1}|\mathcal{R}_{\emph{\textbf{p}}}\varphi(x-\emph{\textbf{p}}\emph{\textbf{k}}_{l})-\mathcal{R}_{\emph{\textbf{p}}}\varphi(x_{n}-\emph{\textbf{p}}\emph{\textbf{k}})|^{2}dx
\ (\ref{cxz} A) \\
\displaystyle \leq\sum^{\#E}_{j=1}|d_{\emph{\textbf{k}}_{j}}|^{2}\#E\frac{C_{1,\emph{\textbf{p}}}}{3\#E(L_{\emph{\textbf{p}},2}-L_{\emph{\textbf{p}},1})}K_{\emph{\textbf{p}}}\delta_{\emph{\textbf{p}}}
\ (\ref{cxz} B)\\
\displaystyle \leq\frac{C_{1,\emph{\textbf{p}}}}{3}\sum^{\#E}_{j=1}|d_{\emph{\textbf{k}}_{j}}|^{2}, \ (\ref{cxz} C)
\end{array}
\end{align}
where
(\ref{cxz}$A$) is from the definition of $h_{l}(x)$, (\ref{cxz}$B$)
is from \eqref{dkkds} and \eqref{KPPP}, and (\ref{cxz}$C$) is from  $K_{\emph{\textbf{p}}}\delta_{\emph{\textbf{p}}}\leq L_{\emph{\textbf{p}},2}-L_{\emph{\textbf{p}},1}.$
Then
\begin{align}\label{1hbcxzcx}\begin{array}{lllllllll}
\displaystyle (\int^{L_{\emph{\textbf{p}},2}}_{L_{\emph{\textbf{p}},1}}|\sum^{\#E}_{l=1}d_{\emph{\textbf{k}}_{l}}h_{l}(x)|^{2}dx)^{1/2}&\geq
\displaystyle-(\int^{L_{\emph{\textbf{p}},2}}_{L_{\emph{\textbf{p}},1}}|\sum^{\#E}_{l=1}d_{\emph{\textbf{k}}_{l}}(\mathcal{R}_{\emph{\textbf{p}}}\varphi(x-\emph{\textbf{p}}\emph{\textbf{k}}_{l})-h_{l}(x))|^{2}dx)^{1/2} \ \ (\ref{1hbcxzcx} A)\\
&\displaystyle+(\int^{L_{\emph{\textbf{p}},2}}_{L_{\emph{\textbf{p}},1}}|\sum^{\#E}_{l=1}d_{\emph{\textbf{k}}_{l}}\mathcal{R}_{\emph{\textbf{p}}}\varphi(x-\emph{\textbf{p}}\emph{\textbf{k}}_{l})|^{2}dx)^{1/2}\\
&\displaystyle \geq(1-\sqrt{1/3})\sqrt{C_{1,\emph{\textbf{p}}}}(\sum^{\#E}_{l=1}|d_{\emph{\textbf{k}}_{l}}|^{2})^{1/2},\ \ (\ref{1hbcxzcx} B)
\end{array}
\end{align}
where  (\ref{1hbcxzcx}$A$) is from the triangle inequality, and
(\ref{1hbcxzcx}$B$) is from  \eqref{ffttt} and \eqref{cxz}.
Then  for any $(d_{\emph{\textbf{k}}_{1}}, \ldots, d_{\emph{\textbf{k}}_{\#E}})\neq \textbf{0}$ we have
\begin{align}\label{hbcx}\begin{array}{lllllllll}
\displaystyle 0<C_{1,\emph{\textbf{p}}}(1-\sqrt{1/3})^{2}\sum^{\#E}_{l=1}|d_{\emph{\textbf{k}}_{l}}|^{2}&\displaystyle \leq\int^{L_{\emph{\textbf{p}},2}}_{L_{\emph{\textbf{p}},1}}|\sum^{\#E}_{l=1}d_{\emph{\textbf{k}}_{l}}h_{l}(x)|^{2}dx\\
&=\displaystyle\sum^{K_{\emph{\textbf{p}}}}_{j=1}\int^{x_{j+1}}_{x_j}|\sum^{\#E}_{l=1}d_{\emph{\textbf{k}}_{l}}h_{l}(x)|^{2}dx\\
&=\displaystyle\sum^{K_{\emph{\textbf{p}}}}_{j=1}|\sum^{\#E}_{l=1}d_{\emph{\textbf{k}}_{l}}\mathcal{R}_{\emph{\textbf{p}}}\varphi(x_{j}-p\emph{\textbf{k}}_{l})|^{2}.
\end{array}
\end{align}
By \eqref{hbcx},  we conclude that
 there exists $X_{\emph{\textbf{p}}}:=\{x_{j_{1}}, \ldots, x_{j_{\#E}}\}\subseteq Y_{\emph{\textbf{p}}}$
such that the corresponding  $\#E\times \#E$ matrix
\begin{align}\label{fgvcb}
A_{\varphi, \emph{\textbf{p}},X_{\emph{\textbf{p}}}}=\left(\begin{array}{cccccccccc}\mathcal{R}_{\emph{\textbf{p}}}\varphi(x_{j_{1}}-\emph{\textbf{p}}\emph{\textbf{k}}_{1})&
\mathcal{R}_{\emph{\textbf{p}}}\varphi(x_{j_{1}}-\emph{\textbf{p}}\emph{\textbf{k}}_{2})&\cdots&\mathcal{R}_{\emph{\textbf{p}}}\varphi(x_{j_{1}}-\emph{\textbf{p}}\emph{\textbf{k}}_{\#E})\\
\mathcal{R}_{\emph{\textbf{p}}}\varphi(x_{j_{2}}-\emph{\textbf{p}}\emph{\textbf{k}}_{1})&
\mathcal{R}_{\emph{\textbf{p}}}\varphi(x_{j_{2}}-\emph{\textbf{p}}\emph{\textbf{k}}_{2})&\cdots&\mathcal{R}_{\emph{\textbf{p}}}\varphi(x_{j_{2}}-\emph{\textbf{p}}\emph{\textbf{k}}_{\#E})\\
\vdots&\vdots&\ddots&\vdots\\
\mathcal{R}_{\emph{\textbf{p}}}\varphi(x_{j_{\#E}}-\emph{\textbf{p}}\emph{\textbf{k}}_{1})&
\mathcal{R}_{\emph{\textbf{p}}}\varphi(x_{j_{\#E}}-\emph{\textbf{p}}\emph{\textbf{k}}_{2})&\cdots&\mathcal{R}_{\emph{\textbf{p}}}\varphi(x_{j_{\#E}}-\emph{\textbf{p}}\emph{\textbf{k}}_{\#E})
\end{array}\right)
\end{align}
is invertible.
Now  by Theorem \ref{theorem123},
the target function $f$ can be determined uniquely by its Radon (w.r.t $\emph{\textbf{p}}$) samples
at $X_{\emph{\textbf{p}}}$. Specifically, the vector $(c_{\emph{\textbf{p}}\emph{\textbf{k}}_{1}}, \ldots, c_{\emph{\textbf{p}}\emph{\textbf{k}}_{\#E}})^{T}$
can be determined by
%
\begin{align}\label{huiyong}(c_{\emph{\textbf{p}}\emph{\textbf{k}}_{1}}, \ldots, c_{\emph{\textbf{p}}\emph{\textbf{k}}_{\#E}})^{T}=A^{-1}_{\varphi, \emph{\textbf{p}}, X_{\emph{\textbf{p}}}}(\mathcal{R}_{\emph{\textbf{p}}}f(x_{j_{1}}), \ldots, \mathcal{R}_{\emph{\textbf{p}}}f(x_{j_{\#E}}))^{T}.\end{align}
This completes the proof.
\end{proof}

In what follows we explain  why  the condition $\nu_{2}(\varphi)>1/2$ in Theorem \ref{yibanshengcyuan} is required.

\begin{rem}
Since $\nu_{2}(\varphi)>1/2$, by   Proposition \ref{lianxuxing} (2)
we conclude that $\mathcal{R}_{\emph{\textbf{p}}}\varphi$ is continuous.
If such a condition is not satisfied, then $\mathcal{R}_{\emph{\textbf{p}}}\varphi$
may be discontinuous for some $\emph{\textbf{p}}.$ As in Remark \ref{kenenglianxu}, let  $\varphi(x_{1}, x_{2})=\chi_{(0,1]}(x_{1})\chi_{(0,1]}(x_{2})$.
Through the direct calculation we have $\widehat{\varphi}(\xi_{1}, \xi_{2})=\frac{1-e^{-\texttt{i}\xi_{1}}}{\texttt{i}\xi_{1}}
\frac{1-e^{-\texttt{i}\xi_{2}}}{\texttt{i}\xi_{2}}$. By the Sobolev smoothness definition  in subsection \ref{sobolevsmoothness} one can check that
$\nu_{2}(\varphi)=1/2$. If $\emph{\textbf{p}}=(1, 0)$ or $(0, 1)$ then $\mathcal{R}_{\emph{\textbf{p}}}\varphi=\chi_{(0, 1]}$ which
is discontinuous. As a result, there may not exist $\delta_{\emph{\textbf{p}}}$ such that \eqref{dkkds} holds.
\end{rem}

\begin{rem}\label{remmak1234}
Define the $\#E\times\#E$ Gram matrix
\begin{align}\label{Grammatrix}
G_{\varphi, \emph{\textbf{p}}}=\Big(\langle \mathcal{R}_{\emph{\textbf{p}}}\varphi(\cdot-\emph{\textbf{p}}\emph{\textbf{k}}_{j}),
\mathcal{R}_{\emph{\textbf{p}}}\varphi(\cdot-\emph{\textbf{p}}\emph{\textbf{k}}_{n})\rangle\Big)^{\#E}_{j,n=1},
\end{align}
where the inner product  $\langle \mathcal{R}_{\emph{\textbf{p}}}\varphi(\cdot-\emph{\textbf{p}}\emph{\textbf{k}}_{j}),
\mathcal{R}_{\emph{\textbf{p}}}\varphi(\cdot-\emph{\textbf{p}}\emph{\textbf{k}}_{n})\rangle=\int_{\mathbb{R}}\mathcal{R}_{\emph{\textbf{p}}}\varphi(x-\emph{\textbf{p}}\emph{\textbf{k}}_{j})
\overline{\mathcal{R}_{\emph{\textbf{p}}}\varphi}(x-\emph{\textbf{p}}\emph{\textbf{k}}_{n})dx.$
Then \eqref{ffttt} or \eqref{ffttt134} is equivalent to
\begin{align} \label{KHJJ} C_{1, \emph{\textbf{p}}}\|(d_{\emph{\textbf{k}}_{1}}, \ldots, d_{\emph{\textbf{k}}_{\#E}})\|^{2}_{2}\leq(d_{\emph{\textbf{k}}_{1}}, \ldots, d_{\emph{\textbf{k}}_{\#E}})G_{\varphi, \emph{\textbf{p}}}
(d_{\emph{\textbf{k}}_{1}}, \ldots, d_{\emph{\textbf{k}}_{\#E}})^{\ast}\leq C_{2, \emph{\textbf{p}}}\|(d_{\emph{\textbf{k}}_{1}}, \ldots, d_{\emph{\textbf{k}}_{\#E}})\|^{2}_{2},\end{align}
where $D^{\ast}$ is the conjugate and transpose of a matrix $D$.
Note that $G_{\varphi, \emph{\textbf{p}}}$ is a Hermitian matrix.
Then \eqref{KHJJ} implies that $G_{\varphi, \emph{\textbf{p}}}$
is a positive definite matrix, and consequently
$0<C_{1, \emph{\textbf{p}}}\leq\lambda_{\min}(G_{\varphi, \emph{\textbf{p}}})$
and $\lambda_{\max}(G_{\varphi, \emph{\textbf{p}}})\leq C_{2,\emph{\textbf{p}}}<\infty$, where
$\lambda_{\max}(G_{\varphi, \emph{\textbf{p}}})>0$ and $\lambda_{\min}(G_{\varphi, \emph{\textbf{p}}})>0$
are the maximum and minimum eigenvalues of $G_{\varphi, \emph{\textbf{p}}}$, respectively.
Particularly, in \eqref{dkkds} one can choose
\begin{align}\label{JKKK} C_{1,\emph{\textbf{p}}}=\lambda_{\min}(G_{\varphi, \emph{\textbf{p}}}).\end{align}
\end{rem}

The following states that if $\varphi\in C^{1}(\mathbb{R}^{2})$ then $\delta_{\emph{\textbf{p}}}$
in the proof of Theorem \ref{yibanshengcyuan} can be chosen explicitly.
Consequently, the SA Radon sampling point set $X_{\emph{\textbf{p}}}$ in
Theorem \ref{yibanshengcyuan} can be constructed explicitly.

\begin{theo}\label{xianshibiaoshi}
Let the compactly supported  generator  $\varphi\in C^{1}(\mathbb{R}^{2})$ such that $\widehat{\varphi}(\textbf{0})\neq0$ and the target function $f\in V(\varphi, \mathbb{Z}^2)$. As in Theorem \ref{yibanshengcyuan}
suppose that   $\hbox{supp}(\varphi) \subseteq[N_{1}, M_{1}]\times [N_{2}, M_{2}]$ and
$\hbox{supp}(f)\subseteq[a_{1}, b_{1}]\times [a_{2}, b_{2}]$. Define two sets
$$ E=\big\{\big[\lceil a_{1}-M_{1}\rceil, \lfloor b_{1}-N_{1}\rfloor\big]\times
\big[\lceil a_{2}-M_{2}\rceil, \lfloor b_{2}-N_{2}\rfloor\big]\big\}\cap \mathbb{Z}^{2}$$
and
\begin{align}E^{+}=\left\{
\begin{array}{cccc}
   \emptyset,&\#E=1,\\
  \{\emph{\textbf{x}}-\emph{\textbf{y}}:
\emph{\textbf{x}}\neq\emph{\textbf{y}}\in E\},&\#E>1.
\end{array}
\right.\end{align}
Choose a   direction vector $\emph{\textbf{p}}\in \{(\cos\theta, \sin\theta): \theta\in [0, 2\pi)\}\setminus\hbox{dv}_{\mathcal{N}_{E^{+}}}$,
and correspondingly  denote
\begin{align}\label{G1}\left\{\begin{array}{lllllll}
L_{\emph{\textbf{p}}, 1}=-\sqrt{2}\max\{|N_{i}|, |M_{i}|: i=1, 2\}+\min\{{\emph{\textbf{p}}}\emph{\textbf{k}}:
\emph{\textbf{k}}\in E\},\\
L_{\emph{\textbf{p}}, 2}=\sqrt{2}\max\{|N_{i}|, |M_{i}|: i=1, 2\}+\max\{{\emph{\textbf{p}}}\emph{\textbf{k}}:
\emph{\textbf{k}}\in E\},\\
\delta_{\emph{\textbf{p}}}=\sqrt{\frac{\lambda_{\min}(G_{\varphi, \emph{\textbf{p}}})}{3\#E(L_{\emph{\textbf{p}},2}-L_{\emph{\textbf{p}},1})}}\Big/\big(2(\|\varphi_{1}\|_{\infty}+\|\varphi_{2}\|_{\infty})\max\{|N_{i}|, |M_{i}|: i=1,2\}\big),\\
K_{\emph{\textbf{p}}}=\lceil\frac{L_{p,2}-L_{p,1}}{\delta_{\emph{\textbf{p}}}}\rceil,
\end{array}\right.
\end{align}
where $\lambda_{\min}(G_{\varphi, \emph{\textbf{p}}})$ is the minimum eigenvalue  of the Gram matrix $G_{\varphi, \emph{\textbf{p}}}$
defined in \eqref{Grammatrix},
$\varphi_{1}(x_{1}, x_{2})$
and $\varphi_{2}(x_{1}, x_{2})$ are the partial derivatives of $\varphi(x_{1}, x_{2})$ w.r.t the variables  $x_{1}$ and $x_{2}$, respectively
such that $$\|\varphi_{1}\|_{\infty}=\max_{(x_{1}, x_{2})\in [N_{1}, M_{1}]\times [N_{2}, M_{2}]}|\varphi_{1}(x_{1}, x_{2})|,
\|\varphi_{2}\|_{\infty}=\max_{(x_{1}, x_{2})\in [N_{1}, M_{1}]\times [N_{2}, M_{2}]}|\varphi_{2}(x_{1}, x_{2})|.$$
Explicitly construct \begin{align}\label{G2}Y_{\emph{\textbf{p}}}=\{x_{k}=L_{\emph{\textbf{p}},1}+\frac{L_{\emph{\textbf{p}},2}-L_{\emph{\textbf{p}},1}}{K_{\emph{\textbf{p}}}}(k-1): k=1, \ldots, K_{\emph{\textbf{p}}}+1\}.\end{align} Then there exists $X_{\emph{\textbf{p}}}=\{x_{i_{1}}, \ldots, x_{i_{\#E}}\}\subseteq
Y_{\emph{\textbf{p}}}$ such that the matrix $A_{\varphi, \emph{\textbf{p}},X_{\emph{\textbf{p}}}}$ in \eqref{fgvcb}
is invertible and consequently, $f$ can be determined uniquely by its SA Radon samples at $X_{\emph{\textbf{p}}}$.
\end{theo}
\begin{proof}
By Remark \ref{remmak1234} \eqref{JKKK},  $C_{1, \emph{\textbf{p}}}$ in   \eqref{KHJJ} can be  chosen as
 $\lambda_{\min}(G_{\varphi, \emph{\textbf{p}}})$. If \eqref{dkkds} holds with  $C_{1, \emph{\textbf{p}}}$ replaced by
 $\lambda_{\min}(G_{\varphi, \emph{\textbf{p}}})$,
then  by the similar procedures (\eqref{cxznew}-\eqref{hbcx}) in the proof of Theorem \ref{yibanshengcyuan}
 one can prove that there exists $X_{\emph{\textbf{p}}}=\{x_{i_{1}}, \ldots, x_{i_{\#E}}\}\subseteq
Y_{\emph{\textbf{p}}}$ such that  $A_{\varphi, \emph{\textbf{p}},X_{\emph{\textbf{p}}}}$ in \eqref{fgvcb}
is invertible. Consequently, $f$ can be determined by \eqref{huiyong}. Therefore, we just need to prove that
\eqref{dkkds} holds.

The SVD of $\emph{\textbf{p}}=(\cos\theta, \sin\theta)$
is $\Sigma V^{T}$ such that
$V=\left(\begin{array}{cccccccccc}
\cos\theta&\sin\theta \\
\sin\theta&-\cos\theta
\end{array}\right)$ and $\Sigma=(1, 0)$.
Since $\varphi\in C^{1}(\mathbb{R}^2)$ is compactly supported,   we have
\begin{align}\label{hxzc}\begin{array}{lllllllll}
|\mathcal{R}_{\emph{\textbf{p}}}\varphi\big(x^{'}-\emph{\textbf{p}}\emph{\textbf{k}}_{l})-\mathcal{R}_{\emph{\textbf{p}}}\varphi\big(x^{''}-\emph{\textbf{p}}\emph{\textbf{k}}_{l}\big)|\\
=\displaystyle\Big|\int_{\mathbb{R}}\varphi\big((x^{'}-\emph{\textbf{p}}\emph{\textbf{k}}_{l})\cos\theta+x_{2}\sin\theta, (x^{'}-\emph{\textbf{p}}\emph{\textbf{k}}_{l})\sin\theta-x_{2}\cos\theta\big) \\
\quad\quad\quad\quad\quad\quad -\varphi((x^{''}-\emph{\textbf{p}}\emph{\textbf{k}}_{l})\cos\theta+x_{2}\sin\theta, (x^{''}-\emph{\textbf{p}}\emph{\textbf{k}}_{l})\sin\theta-x_{2}\cos\theta)dx_{2}\Big| \quad  (\ref{hxzc} A)\\
=\displaystyle\Big|\int^{\max\{|N_{i}|, |M_{i}|: i=1,2\}}_{-\max\{|N_{i}|, |M_{i}|: i=1,2\}}\varphi\big((x^{'}-\emph{\textbf{p}}\emph{\textbf{k}}_{l})\cos\theta+x_{2}\sin\theta, (x^{'}-\emph{\textbf{p}}\emph{\textbf{k}}_{l})\sin\theta-x_{2}\cos\theta\big)\\
\quad\quad\quad\quad\quad\quad\quad-\varphi\big((x^{''}-\emph{\textbf{p}}\emph{\textbf{k}}_{l})\cos\theta+x_{2}\sin\theta, (x^{''}-\emph{\textbf{p}}\emph{\textbf{k}}_{l})\sin\theta-x_{2}\cos\theta\big)dx_{2}\Big|\   (\ref{hxzc} B)\\
\displaystyle\leq|x'-x''|\int^{\max\{|N_{i}|, |M_{i}|: i=1,2\}}_{-\max\{|N_{i}|, |M_{i}|: i=1,2\}}(\|\varphi_{1}\|_{\infty}+\|\varphi_{2}\|_{\infty})dx_{2} \quad (\ref{hxzc} C)\\
\displaystyle=2|x'-x''|(\|\varphi_{1}\|_{\infty}+\|\varphi_{2}\|_{\infty})\max\{|N_{i}|, |M_{i}|: i=1,2\},
\end{array}\end{align}
where the (\ref{hxzc}$A$) and (\ref{hxzc}$B$) are derived from from (\ref{en123456}$A$) and \eqref{uytt}, respectively,
and (\ref{hxzc}$C$) is from the differential mean value theorem.
It is required that $|x'-x''|\leq\delta_{\emph{\textbf{p}}}.$
Then it follows from \eqref{hxzc} that \begin{align} \label{kezhao} |\mathcal{R}_{\emph{\textbf{p}}}\varphi\big(x^{'}-\emph{\textbf{p}}\emph{\textbf{k}}_{l})-\mathcal{R}_{\emph{\textbf{p}}}\varphi\big(x^{''}-\emph{\textbf{p}}\emph{\textbf{k}}_{l}\big)|\leq
2\delta_{\emph{\textbf{p}}}(\|\varphi_{1}\|_{\infty}+\|\varphi_{2}\|_{\infty})\max\{|N_{i}|, |M_{i}|: i=1,2\}.
\end{align}
Now by \eqref{kezhao} we can choose
\begin{align}\notag \delta_{\emph{\textbf{p}}}=\sqrt{\frac{\lambda_{\min}(G_{\varphi, \emph{\textbf{p}}})}{3\#E(L_{\emph{\textbf{p}},2}-L_{\emph{\textbf{p}},1})}}\Big/\big(2(\|\varphi_{1}\|_{\infty}+\|\varphi_{2}\|_{\infty})\max\{|N_{i}|, |M_{i}|: i=1,2\}\big)\end{align}
 such that \eqref{dkkds} holds   with  $C_{1, \emph{\textbf{p}}}$ replaced by
 $\lambda_{\min}(G_{\varphi, \emph{\textbf{p}}})$. The proof is completed.
\end{proof}


\subsection{The second main result: SACT sampling for compactly supported functions in a SIS generated by a vanishing generator $\varphi$}\label{2ndresult}

In this subsection suppose that the generator $\varphi$ is  vanishing, namely,  $\widehat{\varphi}(\textbf{0})=\int_{\mathbb{R}^{2}}\varphi(\emph{\textbf{x}})d\emph{\textbf{x}}=0.$



\begin{theo}\label{yibanshengcyuanYY}
Suppose that $\varphi\in L^2(\mathbb{R}^{2})$ is compactly supported such that  $\hbox{supp}(\varphi)$
$\subseteq[N_{1}, M_{1}]\times [N_{2}, M_{2}]$,
the system $\{\varphi(\cdot-\emph{\textbf{k}}): \emph{\textbf{k}}\in \mathbb{Z}^{2}\}$
is linearly independent, and

(i) the  Sobolev smoothness  $\nu_{2}(\varphi)>1/2$,

(ii)  $\widehat{\varphi}(\textbf{0})=
\int_{\mathbb{R}^{2}}\varphi(\emph{\textbf{x}})d\emph{\textbf{x}}=0$ (vanishing property). \\
Moreover, as previously  suppose that $f\in V(\varphi, \mathbb{Z}^2)$ is an arbitrary target function
such that $\hbox{supp}(f)\subseteq[a_{1}, b_{1}]\times [a_{2}, b_{2}]$.
Define
$$ E=\big\{\big[\lceil a_{1}-M_{1}\rceil, \lfloor b_{1}-N_{1}\rfloor\big]\times
\big[\lceil a_{2}-M_{2}\rceil, \lfloor b_{2}-N_{2}\rfloor\big]\big\}\cap \mathbb{Z}^{2}$$
and
\begin{align}E^{+}=\left\{
\begin{array}{cccc}
   \emptyset,&\#E=1,\\
  \{\emph{\textbf{x}}-\emph{\textbf{y}}:
\emph{\textbf{x}}\neq\emph{\textbf{y}}\in E\},&\#E>1.
\end{array}
\right.\end{align}
Then for any direction vector $\emph{\textbf{p}}\in \hbox{DV}_{\widehat{\varphi}}\setminus\hbox{dv}_{\mathcal{N}_{E^{+}}}$,
there exists a sampling set $X_{\emph{\textbf{p}}}\subseteq \mathbb{R}$ having the cardinality $\#X_{\emph{\textbf{p}}}=\#E$ such that
$f$ can be determined uniquely by its SA Radon (w.r.t $\emph{\textbf{p}}$) samples at $X_{\emph{\textbf{p}}}$,
where $\hbox{DV}_{\widehat{\varphi}}$ and $\hbox{dv}_{\mathcal{N}_{E^{+}}}$ are defined via  Definitions \ref{flagvector} and
\ref{crictial}.
\end{theo}
\begin{proof}
Denote $E$ by $\{\emph{\textbf{k}}_{1}, \ldots, \emph{\textbf{k}}_{\#E}\}$.
It has been proved in the proof of Theorem \ref{yibanshengcyuan} that $\#\hbox{dv}_{\mathcal{N}_{E^{+}}}<\infty$.
Now by  Proposition  \ref{lebesguecedu} (1)  one can prove that $\hbox{DV}_{\widehat{\varphi}}\setminus\hbox{dv}_{\mathcal{N}_{E^{+}}}$
is not empty.
It follows from Proposition   \ref{lebesguecedu} (2) that $\widehat{\mathcal{R}_{\emph{\textbf{p}}}\varphi}\not\equiv0 $  for any direction vector $\emph{\textbf{p}}\in \hbox{DV}_{\widehat{\varphi}}$.
Moreover, as in the proof of Theorem \ref{yibanshengcyuan}   one can prove  that
for any direction vector $\emph{\textbf{p}}  \in \hbox{DV}_{\widehat{\varphi}}\setminus\hbox{dv}_{\mathcal{N}_{E^{+}}}$
the system $\{\mathcal{R}_{\emph{\textbf{p}}}\varphi(\cdot-\emph{\textbf{p}}\emph{\textbf{k}}_{l}): l=1, \ldots, \#E\}$
is linearly independent. Through the similar procedures of the proof of Theorem \ref{yibanshengcyuan}, one can prove there exists  a
sampling set $X_{\emph{\textbf{p}}}$ such that $\#X_{\emph{\textbf{p}}}=\#E$ and
$f$ can be determined uniquely by its SA Radon  samples at $X_{\emph{\textbf{p}}}$.
\end{proof}

If $\varphi\in C^{1}(\mathbb{R}^{2})$, by the similar proof of
Theorem \ref{xianshibiaoshi} one can prove the following result.

\begin{prop}\label{hbvcxz}
If $\varphi\in C^{1}(\mathbb{R}^{2})$, then for $\emph{\textbf{p}}\in \hbox{DV}_{\widehat{\varphi}}\setminus\hbox{dv}_{\mathcal{N}_{E^{+}}}$
the sampling point set $X_{\emph{\textbf{p}}}$ in Theorem \ref{yibanshengcyuanYY}  can be constructed explicitly  through the similar procedures in
Theorem \ref{xianshibiaoshi} \eqref{G1} and \eqref{G2}.
\end{prop}

\subsection{A remark from the perspective of smoothness}
For the sampling-based  recovery in the  SIS $V(\varphi, \mathbb{Z}^2)$, by Aldroubi and   Gr\"{o}chenig \cite{Fienup289} it is
required that $\varphi$ is continuous. That is, if $\varphi$ is discontinuous then the sampling in
$V(\varphi, \mathbb{Z}^2)$ is not well-defined. However, it follows from Remark \ref{kenenglianxu} that
$\mathcal{R}_{\emph{\textbf{p}}}\varphi$ may be continuous even though $\varphi$ is discontinuous.
In this sense, our SA Radon sample-based reconstruction   provides a new perspective for the reconstruction of
functions in $V(\varphi, \mathbb{Z}^2)$  where $\varphi$ is discontinuous.

\section{Pairs of   $(\varphi, \emph{\textbf{p}})$ such that  the corresponding  SACT can be achieved  by the sampling set $\{\emph{\textbf{p}}\emph{\textbf{k}}_{1}, \ldots, \emph{\textbf{p}}\emph{\textbf{k}}_{\#E}\}$} \label{moreq23}
\subsection{Motivation }
If the Sobolev smoothness $\nu_{2}(\varphi)>1/2$, it has been proved in  Theorems  \ref{yibanshengcyuan}
  and \ref{yibanshengcyuanYY} that there exists a sampling set $X_{\emph{\textbf{p}}}$ such that the target function  $f\in V(\varphi, \mathbb{Z}^2)$
can be determined uniquely by its SA samples at $X_{\emph{\textbf{p}}}.$ Moreover,
if $\varphi\in C^{1}(\mathbb{R}^2)$  then it is stated in  Theorem \ref{xianshibiaoshi} and Proposition \ref{hbvcxz}
 that $X_{\emph{\textbf{p}}}$ can be constructed explicitly. On the other hand,  however,  it follows from Remark \ref{continuous123}
that $\nu_{2}(\varphi)>1/2$ does not necessarily imply that $\varphi\in C^{1}(\mathbb{R}^2)$.
Then a natural problem is, without the $C^{1}$ condition how can one explicitly construct   the
the sampling set $X_{\emph{\textbf{p}}}.$ Before introducing our scheme, let us
recall   \eqref{9xy0876} as \begin{align}\label{mubiaohanshu}\mathcal{R}_{\emph{\textbf{p}}}f=\sum^{\#E}_{l=1}c_{\emph{\textbf{k}}_{l}}\mathcal{R}_{\emph{\textbf{p}}}\varphi(\cdot-\emph{\textbf{p}}\emph{\textbf{k}}_{l}).\end{align} Note that $X_{\emph{\textbf{p}}}$ in Theorem \ref{xianshibiaoshi} and Proposition   \ref{hbvcxz}
is not necessarily  $\{\emph{\textbf{p}}\emph{\textbf{k}}_{1}, \ldots, \emph{\textbf{p}}\emph{\textbf{k}}_{\#E}\}$.
Naturally, one asks:

under what condition  on the pair  $(\varphi, \emph{\textbf{p}})$,  can $f$  be determined uniquely by its  SA samples at
$\{\emph{\textbf{p}}\emph{\textbf{k}}_{1}, \ldots, \emph{\textbf{p}}\emph{\textbf{k}}_{\#E}\}$?
\\
If such a determination can be achieved  then
compared with those  in Theorem \ref{xianshibiaoshi} and Proposition \ref{hbvcxz}, it
is more efficient  to conduct SACT of $f$ since we   do  not require to consider the sampling set.

We quickly describe  the structure of this section. In subsection \ref{3rdresult}
we give a condition on the pair $(\varphi, \emph{\textbf{p}})$ such that the above determination can be achieved.
We also  address the determination in subsections \ref{neirong2} and \ref{vanishingqingkuang} for the case  that
$\varphi$ being positive definite.


%
%
\subsection{The third main result: a condition on  $(\varphi, \emph{\textbf{p}})$ such that  $\{\emph{\textbf{p}}\emph{\textbf{k}}_{1}, \ldots, \emph{\textbf{p}}\emph{\textbf{k}}_{\#E}\}$ is   eligible for the SACT} \label{3rdresult}
From the perspective of the signs of the real and imaginary parts of $\widehat{\mathcal{R}_{\emph{\textbf{p}}}\varphi}$,
a condition is given in the following theorem   such that the sampling set   $X_{\emph{\textbf{p}}}=\{\emph{\textbf{p}}\emph{\textbf{k}}_{1}, \ldots, \emph{\textbf{p}}\emph{\textbf{k}}_{\#E}\}$ is eligible for the SACT. Incidentally, for  $0\neq y\in \mathbb{R}$ its sign  $\hbox{sgn}(y)$
takes $1$ and $-1$ for $y>0$ and $y<0$, respectively.
For a function $0\not\equiv g: \mathbb{R}\longrightarrow \mathbb{R}$ we say that its sign function  $\hbox{sgn}(g(x))$ is \textbf{unchanged} if $g(x)\geq 0$
for any $x\in \mathbb{R}$  (or $g(x)\leq0$ for any $x\in \mathbb{R}$).

\begin{theo}\label{HHKK}
As previously, suppose that the generator  $\varphi\in L^{2}(\mathbb{R}^{2})$
satisfies  $\hbox{supp}(\varphi)$
$\subseteq[N_{1}, M_{1}]\times
[N_{2}, M_{2}]$, and   $f\in V(\varphi, \mathbb{Z}^{2})$ is an arbitrary target function such that $ \hbox{supp}(f)\subseteq [a_{1}, b_{1}]\times [a_{2}, b_{2}]$.
Additionally, suppose that
$\emph{\textbf{p}}=(\cos\theta, \sin\theta)$ is a direction vector such that $\mathcal{R}_{\emph{\textbf{p}}}\varphi$ is continuous.
 Define   $ E=\{\emph{\textbf{k}}_{1}, \ldots, \emph{\textbf{k}}_{\#E}\}=\big\{\big[\lceil a_{1}-M_{1}\rceil, \lfloor b_{1}-N_{1}\rfloor\big]\times
\big[\lceil a_{2}-M_{2}\rceil, \lfloor b_{2}-N_{2}\rfloor\big]\big\}\cap \mathbb{Z}^{2}$.
If
 $\widehat{\mathcal{R}_{\emph{\textbf{p}}}\varphi}=
\widehat{\mathcal{R}_{\emph{\textbf{p}}}\varphi}_{\Re}+\texttt{i}\widehat{\mathcal{R}_{\emph{\textbf{p}}}\varphi}_{\Im}$
satisfies the following  item (i) or (ii), and  $E$ satisfies item (iii):

(i) the real part $\widehat{\mathcal{R}_{\emph{\textbf{p}}}\varphi}_{\Re}\not\equiv0 $
and  its sign function  $\hbox{sgn}(\widehat{\mathcal{R}_{\emph{\textbf{p}}}\varphi}_{\Re})$ is unchanged;

(ii) the imaginary part $\widehat{\mathcal{R}_{\emph{\textbf{p}}}\varphi}_{\Im}\not\equiv0$ and  its sign function  $\hbox{sgn}(\widehat{\mathcal{R}_{\emph{\textbf{p}}}\varphi}_{\Im})$ is unchanged;

 (iii) if $\#E>1$ then    $\emph{\textbf{p}}\emph{\textbf{k}}_{l}\neq \emph{\textbf{p}} \emph{\textbf{k}}_{n}$ for any $l\neq n$;\\
then the matrix
 $A_{\varphi,\emph{\textbf{p}}, X_{\emph{\textbf{p}}}}$ defined via   \eqref{qixiang1234} is invertible
 where $X_{\emph{\textbf{p}}}=\{\emph{\textbf{p}}\emph{\textbf{k}}_{1}, \ldots, \emph{\textbf{p}}\emph{\textbf{k}}_{\#E}\}$;\\
Consequently,
$f$ can be determined uniquely by its
SA Radon samples at $X_{\emph{\textbf{p}}}$.
\end{theo}
\begin{proof}
The requirement for  the  continuity of $\mathcal{R}_{\emph{\textbf{p}}}\varphi$ in  Theorem \ref{theorem123} is satisfied here.
If the corresponding matrix
$A_{\varphi,\emph{\textbf{p}},X_{\emph{\textbf{p}}}}$ defined via  \eqref{qixiang1234} is invertible,
then it follows from Theorem \ref{theorem123} that  $f$ can be determined uniquely by its
SA Radon  samples at $X_{\emph{\textbf{p}}}=\{\emph{\textbf{p}}\emph{\textbf{k}}_{1}, \ldots, \emph{\textbf{p}}\emph{\textbf{k}}_{\#E}\}$.
We next prove that $A_{\varphi,\emph{\textbf{p}},X_{\emph{\textbf{p}}}}$ is invertible.
For any nonzero  vector $(\alpha_{1}, \ldots, \alpha_{\#E})^{T}$
$\in \mathbb{C}^{\#E}$ we have
\begin{align}\label{zhengdingold}\begin{array}{lllllllll}
\displaystyle \sum^{\#E}_{j=1}\sum^{\#E}_{n=1}\alpha_{j}\overline{\alpha}_{n}\mathcal{R}_{\emph{\textbf{p}}}\varphi(\emph{\textbf{p}}\emph{\textbf{k}}_{j}-\emph{\textbf{p}}\emph{\textbf{k}}_{n})&=
\displaystyle\frac{1}{2\pi}\sum^{\#E}_{j=1}\sum^{\#E}_{n=1}\alpha_{j}\overline{\alpha}_{n}\int_{\mathbb{R}}\widehat{\mathcal{R}_{\emph{\textbf{p}}}\varphi}(\xi)e^{\texttt{i}(\emph{\textbf{p}}\emph{\textbf{k}}_{j}-\emph{\textbf{p}}\emph{\textbf{k}}_{n})\xi}d\xi\\
&\displaystyle=\frac{1}{2\pi}\int_{\mathbb{R}}\widehat{\mathcal{R}_{\emph{\textbf{p}}}\varphi}(\xi)|\sum^{\#E}_{j=1}
\alpha_{j}e^{\texttt{i}\emph{\textbf{p}}\emph{\textbf{k}}_{j}\xi}|^{2}d\xi \quad  (\ref{zhengdingold} \hbox{A})
\end{array}
\end{align}
where  (\ref{zhengdingold}\hbox{A}) is derived from the  quadratic form $\sum^{\#E}_{j=1}\sum^{\#E}_{n=1}\alpha_{j}e^{\texttt{i}\emph{\textbf{p}}\emph{\textbf{k}}_{j}\xi}\overline{\alpha}_{n}
e^{-\texttt{i}\emph{\textbf{p}}\emph{\textbf{k}}_{n}\xi}=|\sum^{\#E}_{j=1}
\alpha_{j}e^{\texttt{i}\emph{\textbf{p}}\emph{\textbf{k}}_{j}\xi}|^{2}.$
By $\widehat{\mathcal{R}_{\emph{\textbf{p}}}\varphi}(\xi)=\widehat{\mathcal{R}_{\emph{\textbf{p}}}\varphi}_{\Re}(\xi)+\texttt{i}
\widehat{\mathcal{R}_{\emph{\textbf{p}}}\varphi}_{\Im}(\xi)$, \eqref{zhengdingold} can be further expressed as
\begin{align}\label{zhengding}\begin{array}{lllllllll}
\displaystyle\sum^{\#E}_{j=1}\sum^{\#E}_{n=1}\alpha_{j}\overline{\alpha}_{n}\mathcal{R}_{\emph{\textbf{p}}}\varphi(\emph{\textbf{p}}\emph{\textbf{k}}_{j}-\emph{\textbf{p}}\emph{\textbf{k}}_{n})
&\displaystyle=\frac{1}{2\pi}\int_{\mathbb{R}}\widehat{\mathcal{R}_{\emph{\textbf{p}}}\varphi}_{\Re}(\xi)|\sum^{\#E}_{j=1}
\alpha_{j}e^{\texttt{i}\emph{\textbf{p}}\emph{\textbf{k}}_{j}\xi}|^{2}d\xi\\
&\displaystyle+\frac{\texttt{i}}{2\pi}\int_{\mathbb{R}}\widehat{\mathcal{R}_{\emph{\textbf{p}}}\varphi}_{\Im}(\xi)|\sum^{\#E}_{j=1}
\alpha_{j}e^{\texttt{i}\emph{\textbf{p}}\emph{\textbf{k}}_{j}\xi}|^{2}d\xi.
\end{array}
\end{align}
Since $\varphi\in L^2(\mathbb{R}^{2})$ is compactly supported,  it follows from Lemma \ref{pouy} that  $\mathcal{R}_{\emph{\textbf{p}}}\varphi$
is compactly supported as well and belongs to $L^{2}(\mathbb{R})$. Then $\widehat{\mathcal{R}_{\emph{\textbf{p}}}\varphi}\in C^{\infty}(\mathbb{R})$.
Item (i) or (ii) implies that    $\widehat{\mathcal{R}_{\emph{\textbf{p}}}\varphi}\not\equiv0$.
 Without loss of generality it is assumed that $\widehat{\mathcal{R}_{\emph{\textbf{p}}}\varphi}_{\Re}\not\equiv0$ and
$\widehat{\mathcal{R}_{\emph{\textbf{p}}}\varphi}_{\Re}\geq0$. By  the continuity of $\widehat{\mathcal{R}_{\emph{\textbf{p}}}\varphi}_{\Re}$ there is $\delta>0$ and $\zeta\in \mathbb{R}$ such that
\begin{align}\label{bubianhao}\widehat{\mathcal{R}_{\emph{\textbf{p}}}\varphi}_{\Re}(\xi)>0\end{align} for any $\xi\in [\zeta-\delta, \zeta+\delta]$.
We next prove that $\{e^{\texttt{i}\emph{\textbf{p}}\emph{\textbf{k}}_{n}\xi}\}^{\#E}_{n=1}$
is linearly independent  on $[\zeta-\delta, \zeta+\delta]$.
If $\#E=1$ then the linear independence is clear.
If $\#E>1$ then it follows from item  (iii) that
$\emph{\textbf{p}}\emph{\textbf{k}}_{n}\neq \emph{\textbf{p}}\emph{\textbf{k}}_{j}$ for any $n\neq j$.
By Lemma \ref{lemma1234} we have   that $\{e^{\texttt{i}\emph{\textbf{p}}\emph{\textbf{k}}_{n}\xi}\}^{\#E}_{n=1}$
is linearly independent  on $[\zeta-\delta, \zeta+\delta]$. Consequently, there exists
$\xi_{0}\in [\zeta-\delta, \zeta+\delta]$ such that for the above  nonzero  vector $(\alpha_{1}, \ldots, \alpha_{\#E})^{T}$
$\in \mathbb{C}^{\#E}$ we have $\sum^{\#E}_{n=1}
\alpha_{n}e^{\texttt{i}\emph{\textbf{p}}\emph{\textbf{k}}_{n}\xi_{0}}\neq0.$ By the continuity of the functions in $\{e^{\texttt{i}\emph{\textbf{p}}\emph{\textbf{k}}_{n}\xi}\}^{\#E}_{n=1}$
we conclude that
\begin{align}\label{ouyang} \int^{\zeta+\delta}_{\zeta-\delta}|\sum^{\#E}_{n=1}
\alpha_{n}e^{\texttt{i}\emph{\textbf{p}}\emph{\textbf{k}}_{n}\xi}|^{2}d\xi>0.\end{align}
%
%
Now it follows from  \eqref{bubianhao} and \eqref{ouyang} that
\begin{align}\label{zhengding1234}\begin{array}{lllllllll}
\displaystyle \frac{1}{2\pi}\int_{\mathbb{R}}\widehat{\mathcal{R}_{\emph{\textbf{p}}}\varphi}_{\Re}(\xi)|\sum^{\#E}_{n=1}
\alpha_{n}e^{\texttt{i}\emph{\textbf{p}}\emph{\textbf{k}}_{n}\xi}|^{2}d\xi&\displaystyle\geq\frac{1}{2\pi}\int^{\zeta+\delta}_{\zeta-\delta}\widehat{\mathcal{R}_{\emph{\textbf{p}}}\varphi}_{\Re}(\xi)|\sum^{\#E}_{n=1}
\alpha_{k}e^{\texttt{i}\emph{\textbf{p}}\emph{\textbf{k}}_{n}\xi}|^{2}d\xi\\
&\displaystyle\geq\frac{1}{2\pi}\min_{\xi\in[\zeta-\delta, \zeta+\delta]}\{\widehat{\mathcal{R}_{\emph{\textbf{p}}}\varphi}_{\Re}(\xi)\} \int^{\zeta+\delta}_{\zeta-\delta}|\sum^{\#E}_{n=1}
\alpha_{n}e^{\texttt{i}\emph{\textbf{p}}\emph{\textbf{k}}_{n}\xi}|^{2}d\xi\\
&>0.
\end{array}
\end{align}
This combining with \eqref{zhengding} leads to that
\begin{align}\label{HJKK}\sum^{\#E}_{j=1}\sum^{\#E}_{n=1}\alpha_{j}\overline{\alpha}_{n}\mathcal{R}_{\emph{\textbf{p}}}\varphi(\emph{\textbf{p}}\emph{\textbf{k}}_{j}-\emph{\textbf{p}}\emph{\textbf{k}}_{n})>0.\end{align}
Recall that
\begin{align} \label{BVXCD} \sum^{\#E}_{j=1}\sum^{\#E}_{n=1}\alpha_{j}\overline{\alpha}_{n}\varphi(\emph{\textbf{p}}\emph{\textbf{k}}_{j}-\emph{\textbf{p}}\emph{\textbf{k}}_{n})=(\bar{\alpha}_{1}, \ldots, \bar{\alpha}_{\#E})A_{\varphi,\emph{\textbf{p}},X_{\emph{\textbf{p}}}}(\alpha_{1}, \ldots, \alpha_{\#E})^{T}\end{align}
and $(\alpha_{1}, \ldots, \alpha_{\#E})^{T}$
$\in \mathbb{C}^{\#E}$ is an arbitrary nonzero  vector.
Now it follows from \eqref{HJKK} and \eqref{BVXCD} that
$A_{\varphi,\emph{\textbf{p}},X_{\emph{\textbf{p}}}}$   is invertible.
By Theorem \ref{theorem123}, $f$ can be determined uniquely by its
SA Radon samples at $\{\emph{\textbf{p}}\emph{\textbf{k}}_{1}, \ldots, \emph{\textbf{p}}\emph{\textbf{k}}_{\#E}\}$.
\end{proof}

\subsection{Preliminary on positive (semi-)definite function}\label{moreqe}

\begin{defi}\label{definitionofPD}
 We say that a function $\phi: \mathbb{R}^{d}\longrightarrow \mathbb{C}$   is positive semi-definite   if  for all $N\in \mathbb{N}$, all sets $X=\{\emph{\textbf{x}}_{1}, \emph{\textbf{x}}_{2}, \ldots, \emph{\textbf{x}}_{N}\}\subseteq \mathbb{R}^{d}$, and all vectors  $\textbf{0}\neq (\alpha_{1}, \ldots, \alpha_{N})^{T}\in \mathbb{C}^{N}$, the quadratic form
\begin{align}\label{zhengding00}\begin{array}{llllll}
\sum^{N}_{j=1}\sum^{N}_{k=1}\alpha_{j}\overline{\alpha}_{k}\phi(\emph{\textbf{x}}_{j}-\emph{\textbf{x}}_{k})\\
=(\bar{\alpha}_{1}, \bar{\alpha}_{2}, \ldots, \bar{\alpha}_{N})\left(\begin{array}{cccccccccc}
\phi(\textbf{0})&\phi(\emph{\textbf{x}}_{1}-\emph{\textbf{x}}_{2})&\cdots&\phi(\emph{\textbf{x}}_{1}-\emph{\textbf{x}}_{N})\\
\phi(\emph{\textbf{x}}_{2}-\emph{\textbf{x}}_{1})&\phi(\textbf{0})&\cdots&\phi(\emph{\textbf{x}}_{2}-\emph{\textbf{x}}_{N})\\
\vdots&\vdots&\ddots&\vdots\\
\phi(\emph{\textbf{x}}_{N}-\emph{\textbf{x}}_{1})&\phi(\emph{\textbf{x}}_{N}-\emph{\textbf{x}}_{2})&\cdots&\phi(\textbf{0})
\end{array}\right)\left(\begin{array}{cccccccccc}
\alpha_{1}\\
\alpha_{2}\\
\vdots\\
\alpha_{N}
\end{array}\right)\\
\geq0.
\end{array}
\end{align}
Furthermore, the function $\phi$ is  positive definite if  the above quadratic form is positive
for all $\textbf{0}\neq(\alpha_{1}, \ldots, \alpha_{N})^{T}$.
\end{defi}

The celebrated result on positive semi-definite functions is their characterization    in terms of Fourier transform,
which  was established by Bochner \cite{Bochner}.
It is   as follows.
\begin{lem}\label{fuding}
A continuous function  $\phi: \mathbb{R}^{d}\longrightarrow \mathbb{C}$
is  positive semi-definite if and only if it is the Fourier transform of a finite nonnegative Borel measure $\mu$ on
$\mathbb{R}^{d}$ such that
$\phi(\emph{\textbf{x}})=\int_{\mathbb{R}^{d}}e^{-\texttt{i}\emph{\textbf{x}}\cdot\xi}d\mu(\xi).$
\end{lem}

Based on Lemma \ref{fuding}, Wendland \cite[Theorem 6.11]{Wendlan} established the following tool for checking the positive definite property,  which will be used in Theorems  \ref{ddgo} and \ref{eligiblevectors} for  SACT sampling.

\begin{lem}\label{characerization}
Suppose that $\phi\in L^{1}(\mathbb{R}^{d})$ is continuous. Then $\phi$ is positive definite if and only if
$\phi$ is bounded  and its Fourier transform $\widehat{\phi}$ is nonnegative and nonvanishing.
Here $\widehat{\phi}$ being nonvanishing means  that $\int_{\mathbb{R}^{d}}\widehat{\phi}(\xi)d\xi=(2\pi)^{d/2}\phi(\textbf{0})\neq0.$
\end{lem}

The following remark  concerns on    the determination  of functions by  the positive definite  property.

\begin{rem}\label{6758}
If $\phi: \mathbb{R}^{d}\longrightarrow \mathbb{C}$ is positive definite  and continuous, then   the system   $\{\phi(\cdot-\emph{\textbf{x}}_{k})\}^{N}_{k=1}$ is linearly independent for any set  $X=\{\emph{\textbf{x}}_{1}, \emph{\textbf{x}}_{2}, \ldots, \emph{\textbf{x}}_{N}\}\subseteq \mathbb{R}^{d}$. Moreover, any function $f=\sum^{N}_{k=1}c_{k}\phi(\cdot-\emph{\textbf{x}}_{k})$ can be determined uniquely  by its samples  at $X$.
\end{rem}

\begin{proof}
If $\{\phi(\cdot-\emph{\textbf{x}}_{k})\}^{N}_{k=1}$ is linearly dependent
then there exists a nonzero vector $(\alpha_{1}, \ldots, \alpha_{N})\in \mathbb{C}^N$
such that $\sum^{N}_{k=1}\alpha_{k}\phi(\cdot-\emph{\textbf{x}}_{k})\equiv0$.
Particularly, for any $\emph{\textbf{x}}_{j}\in \{\emph{\textbf{x}}_{1}, \ldots, \emph{\textbf{x}}_{N}\}$ we have
$\sum^{N}_{k=1}\alpha_{k}\phi(\emph{\textbf{x}}_{j}-\emph{\textbf{x}}_{k})=0$.
Then the  quadratic form $$\sum^{N}_{j=1}\sum^{N}_{k=1}\alpha_{j}\overline{\alpha}_{k}\phi(\emph{\textbf{x}}_{j}-\emph{\textbf{x}}_{k})=0.$$
This contradicts with the positive definite property
\begin{align}\label{KLG}\sum^{N}_{j=1}\sum^{N}_{k=1}\alpha_{j}\overline{\alpha}_{k}\phi(\emph{\textbf{x}}_{j}-\emph{\textbf{x}}_{k})>0.\end{align}
Therefore,
$\{\phi(\cdot-\emph{\textbf{x}}_{k})\}^{N}_{k=1}$ is linearly independent.
Additionally,
\begin{align}\label{YTGC}\begin{array}{llllll}
\left(\begin{array}{cccccccccc}
\phi(\textbf{0})&\phi(\emph{\textbf{x}}_{1}-\emph{\textbf{x}}_{2})&\cdots&\phi(\emph{\textbf{x}}_{1}-\emph{\textbf{x}}_{N})\\
\phi(\emph{\textbf{x}}_{2}-\emph{\textbf{x}}_{1})&\phi(\textbf{0})&\cdots&\phi(\emph{\textbf{x}}_{2}-\emph{\textbf{x}}_{N})\\
\vdots&\vdots&\ddots&\vdots\\
\phi(\emph{\textbf{x}}_{N}-\emph{\textbf{x}}_{1})&\phi(\emph{\textbf{x}}_{N}-\emph{\textbf{x}}_{2})&\cdots&\phi(\textbf{0})
\end{array}\right)\left(\begin{array}{cccccccccc}
c_{1}\\
c_{2}\\
\vdots\\
c_{N}
\end{array}\right)=\left(\begin{array}{cccccccccc}
f(\emph{\textbf{x}}_{1})\\
f(\textbf{\emph{x}}_{2})\\
\vdots\\
f(\emph{\textbf{x}}_{N})
\end{array}\right).
\end{array}
\end{align}
Since $\phi$ is positive definite, it follows from \eqref{KLG}
that the above matrix is invertible.
Then the coefficient vector $(c_{1}, c_{2}, \ldots, c_{N})^{T}$ can be determined by the samples $f(\emph{\textbf{x}}_{1}),
f(\emph{\textbf{x}}_{2}), \ldots, f(\emph{\textbf{x}}_{N})$. Recall that $\{\phi(\cdot-\emph{\textbf{x}}_{k})\}^{N}_{k=1}$ is linearly independent.
With $(c_{1}, \ldots, c_{N})^{T}$ at hand, $f$ can be determined uniquely.
\end{proof}

\subsection{The fourth    main result:  pairs of  $(\varphi, \emph{\textbf{p}})$
such that $\{\emph{\textbf{p}}\emph{\textbf{k}}_{1}, \ldots, \emph{\textbf{p}}\emph{\textbf{k}}_{\#E}\}$
is eligible for SACT sampling, where
$\varphi$ is positive definite and nonvanishing.}\label{neirong2}
The following is the main result in this subsection. It applies to the case that $\varphi$ is  positive definite and
 nonvanishing ($\widehat{\varphi}(\textbf{0})\neq0$).

\begin{theo}\label{ddgo}
Suppose that $\varphi\in C(\mathbb{R}^{2})$ is compactly supported and positive definite
such that its Sobolev smoothness  $\nu_{2}(\varphi)>1/2$, $\widehat{\varphi}(\textbf{0})>0$ and   $\hbox{supp}(\varphi)\subseteq[N_{1}, M_{1}]\times
[N_{2}, M_{2}]$. Moreover, the arbitrary target function  $f\in V(\varphi, \mathbb{Z}^2)$
is  compactly supported such that $ \hbox{supp}(f)\subseteq [a_{1}, b_{1}]\times [a_{2}, b_{2}]$.
As previously, define  $ E=\{\emph{\textbf{k}}_{1}, \ldots, \emph{\textbf{k}}_{\#E}\}=\big\{\big[\lceil a_{1}-M_{1}\rceil, \lfloor b_{1}-N_{1}\rfloor\big]\times
\big[\lceil a_{2}-M_{2}\rceil, \lfloor b_{2}-N_{2}\rfloor\big]\big\}\cap \mathbb{Z}^{2},$
and correspondingly \begin{align}E^{+}=\left\{
\begin{array}{cccc}
   \emptyset,&\#E=1,\\
  \{\emph{\textbf{x}}-\emph{\textbf{y}}:
\emph{\textbf{x}}\neq\emph{\textbf{y}}\in E\},&\#E>1.
\end{array}
\right.\end{align}
Then
$f$ can be determined uniquely by its
SA Radon (w.r.t $\emph{\textbf{p}}$) samples at  $\{\emph{\textbf{p}}\emph{\textbf{k}}_{1}, \ldots, $
$\emph{\textbf{p}}\emph{\textbf{k}}_{\#E}\}$, where $\emph{\textbf{p}}$ is an arbitrary  direction vector from
$\{(\cos\theta, \sin\theta): \theta\in [0, 2\pi)\}\setminus\hbox{dv}_{\mathcal{N}_{E^{+}}}$
with $\mathcal{N}_{E^{+}}$ defined in Definition \ref{crictial}.
\end{theo}
\begin{proof}
Recall that it has been proved in the proof of Theorem \ref{yibanshengcyuan} that   $\{(\cos\theta, \sin\theta): \theta\in [0, 2\pi)\}\setminus\hbox{dv}_{\mathcal{N}_{E^{+}}}$
is not empty.
Next we     prove  the following three items.

(1) If $\#E>1$ then  for any   $l\neq n\in \{1, \ldots, \#E\}$ and any direction vector $\emph{\textbf{p}}\in \{(\cos\theta, \sin\theta): \theta\in [0, 2\pi)\}\setminus\hbox{dv}_{\mathcal{N}_{E^{+}}}$, we have $\emph{\textbf{p}}\emph{\textbf{k}}_{l}\neq\emph{\textbf{p}}\emph{\textbf{k}}_{n}$.

(2) For any  direction vector $\emph{\textbf{p}}$, the Radon transform $\mathcal{R}_{\emph{\textbf{p}}}\varphi$
is continuous.

(3) Suppose that  $\emph{\textbf{p}}$ is any   fixed  direction vector. Then we have
$\widehat{\mathcal{R}_{\emph{\textbf{p}}}\varphi}\not\equiv0$ and $\widehat{\mathcal{R}_{\emph{\textbf{p}}}\varphi}\geq0$.\\
Clearly, if the above three items are satisfied then     the requirements in Theorem \ref{HHKK} are satisfied  for  any direction vector $\emph{\textbf{p}}\in \{(\cos\theta, \sin\theta): \theta\in [0, 2\pi)\}\setminus\hbox{dv}_{\mathcal{N}_{E^{+}}}$.
Consequently, it follows from Theorem \ref{HHKK} that
$f$ can be determined uniquely by its
SA Radon (w.r.t $\emph{\textbf{p}}$) samples at  $\{\emph{\textbf{p}}\emph{\textbf{k}}_{1}, \ldots, \emph{\textbf{p}}\emph{\textbf{k}}_{\#E}\}$.

We first prove (1). One can check that, for any
$l\neq n\in \{1, 2, \ldots, \#E\} $ it holds that
$\emph{\textbf{p}}\emph{\textbf{k}}_{l}\neq \emph{\textbf{p}}\emph{\textbf{k}}_{n}$
if and only if $\emph{\textbf{p}}\notin \mathcal{N}_{E^{+}}$. Then for any $\emph{\textbf{p}}\in
\{(\cos\theta, \sin\theta): \theta\in [0, 2\pi)\}\setminus\hbox{dv}_{\mathcal{N}_{E^{+}}}$,
item (1) holds.


Next we prove item (2). Recall that $\nu_{2}(\varphi)>1/2$. Then it follows from
 Proposition \ref{lianxuxing} (2) that $\mathcal{R}_{\emph{\textbf{p}}}\varphi$ is continuous.

Finally, we need to prove item (3).
Since $\varphi\in C(\mathbb{R}^{2})$ is  positive definite, by Lemma \ref{characerization} we have $\widehat{\varphi}\geq0$.
Now for any direction vector $\emph{\textbf{p}}=(\cos\theta, \sin\theta)$ we have that $\widehat{\mathcal{R}_{\emph{\textbf{p}}}\varphi}=\widehat{\varphi}(\emph{\textbf{p}}^{T}\cdot)\geq0$.
Additionally,
$\varphi\in C(\mathbb{R}^{2})$ is compactly supported then $\widehat{\varphi}\in C^{\infty}(\mathbb{R}^{2})$.
This together with
$\widehat{\varphi}(\textbf{0})=\int_{\mathbb{R}^{2}}\varphi(\emph{\textbf{x}})d\emph{\textbf{x}}>0$ leads to that  exists a closed disc  $\mathcal{D}(\textbf{0}, \delta)=
\{\xi\in \mathbb{R}^{2}: \|\xi\|_{2}\leq\delta\}$  such that for any $\xi\in \mathcal{D}(\textbf{0}, \delta)$ we have $\widehat{\varphi}(\xi)>0.$
For any $\gamma\in \mathbb{R}$ such that $|\gamma|\leq\delta$
we have $\gamma \emph{\textbf{p}}^{T}\in U(\textbf{0}, \delta)$ and consequently  $\widehat{\mathcal{R}_{\emph{\textbf{p}}}\varphi}(\gamma)=\widehat{\varphi}(\gamma \emph{\textbf{p}}^{T})>0.$
That is, $\widehat{\mathcal{R}_{\emph{\textbf{p}}}\varphi}\not\equiv0.$

 Now by Theorem \ref{HHKK} the target function
$f$ can be determined uniquely by its
SA Radon   samples at  $\{\emph{\textbf{p}}\emph{\textbf{k}}_{1}, \ldots, \emph{\textbf{p}}\emph{\textbf{k}}_{\#E}\}$.
The proof is completed.
%
%
%
%
%
\end{proof}

\begin{rem}\label{xiangliangdezuoyong}
As addressed in item (3) of the proof of  Theorem \ref{ddgo}, the nonvanishing property $\widehat{\varphi}(\textbf{0})=\int_{\mathbb{R}^{2}}\varphi(\emph{\textbf{x}})d\emph{\textbf{x}}>0$ guarantees that $\widehat{\mathcal{R}_{\emph{\textbf{p}}}\varphi}\not\equiv0$  for any direction vector $\emph{\textbf{p}}.$
Such a property brings great flexibility of  $\emph{\textbf{p}}$
for the SACT sampling.
In what we  follows introduce   a class of  box-splines which are positive and nonvanishing.
\end{rem}

The $m$th cardinal B-spline $B_{m}$ is defined
by
$B_{m}:=\overbrace{\chi_{(0,1]}\star\ldots\star\chi_{(0,1]}}^{m\ \hbox{copies}}$ (c.f. \cite{Wenchangsun0,Wenchangsun}),
where $m\in \mathbb{N}$, as in Remark \ref{continuous123} $\chi_{(0,1]}$ is the characteristic function of $(0,1]$
and $\star$ is the convolution.
Through the  simple calculation  (c.f. \cite{Chui}) we have $\hbox{supp}(B_{m})=(0, m]$, and   \begin{align}
\label{fuliyebiaodashi}\widehat{B_{m}}(\xi)=e^{-\texttt{i} m\xi/2}[\frac{\sin\xi/2}{\xi/2}]^{m}. \end{align}

\begin{rem}\label{guanghuaxingremark}
 For $s<m-1/2$, one can check that
$\int_{\mathbb{R}}|\widehat{B_{m}}(\xi)|^{2}(1+\xi^{2})^{s}d\xi<\infty$.
Then   the Sobolev smoothness $\nu_{2}(B_{m})=m-1/2.$
By Proposition \ref{lianxuxing}, $B_{m}$ is continuous for $m\geq2.$
\end{rem}

\begin{prop}\label{ertyfc}
Through the  tensor product we  define the box-spline $\varphi: \mathbb{R}^{2}\longrightarrow \mathbb{R}$ by
$\varphi(x_{1},  x_{2})=\prod^{2}_{k=1}B_{2n_{k}}(x_{k}+n_{k}), \ n_{k}\in \mathbb{N}.$
Then $\varphi$ is   compactly supported, continuous and  positive definite    such that $\widehat{\varphi}(\textbf{0})=\int_{\mathbb{R}^{2}}\varphi(\emph{\textbf{x}})d\emph{\textbf{x}}>0.$
Moreover, the Sobolev smoothness  $\nu_{2}(\varphi)\geq1$. Consequently, it satisfies the requirement in  Theorem \ref{ddgo}.
\end{prop}
\begin{proof}
By Remark \ref{guanghuaxingremark}, both $B_{2n_{1}}$
and $B_{2n_{2}}$ are continuous. So are
$B_{2n_{1}}(\cdot+n_{1})$
and $B_{2n_{2}}(\cdot+n_{2})$. Then their tensor product $\varphi$ is also continuous.
Through the direct calculation one can check that  $ \widehat{\varphi}(\xi_{1}, \xi_{2})=
\prod^{2}_{k=1}\Big[\frac{\sin\xi_{k}/2}{\xi_{k}/2}\Big]^{2n_{k}},$
and it follows from $\hbox{supp}(B_{2n_{i}})=(0, 2n_{i}]$ that   $\hbox{supp}(\varphi)=(-n_{1}, n_{1}]\times (-n_{2}, n_{2}]$.
Clearly, $\widehat{\varphi}\geq0$ and $\widehat{\varphi}(\textbf{0})=\int_{\mathbb{R}^{2}}\varphi(\emph{\textbf{x}})d\emph{\textbf{x}}=1$.
As in  (\ref{zhengdingold}\hbox{A}), for any nonzero  $(\alpha_{1}, \ldots, \alpha_{N})\in \mathbb{C}^{N}$
and any set  $\{\emph{\textbf{x}}_{1}, \ldots, \emph{\textbf{x}}_{N}\}\subseteq \mathbb{R}^{2}$
one can check that
\begin{align}\label{BVCNN}\begin{array}{lllllllll}
\displaystyle \sum^{N}_{j=1}\sum^{N}_{k=1}\alpha_{j}\overline{\alpha}_{k}\varphi(\emph{\textbf{x}}_{j}-\emph{\textbf{x}}_{k})
&\displaystyle=\frac{1}{2\pi}\int_{\mathbb{R}^{2}}\widehat{\varphi}(\xi)|\sum^{N}_{k=1}
\alpha_{k}e^{\texttt{i}\emph{\textbf{x}}_{k}\cdot\xi}|^{2}d\xi.
\end{array}
\end{align}
From this and  $\widehat{\varphi}\geq0$ we have  that   $\varphi$ is positive semi-definite.
By Lemma \ref{lemma1234}, the set of continuous  functions $\{e^{\texttt{i}\emph{\textbf{x}}_{k}\cdot\xi}: k=1, \ldots, N\}$ are  linearly independent. Since $\varphi\in C(\mathbb{R}^2)$ is compactly supported then $\widehat{\varphi}\in C^{\infty}(\mathbb{R}^2)$.
Now combining the continuities of $\widehat{\varphi}$ and $\{e^{\texttt{i}\emph{\textbf{x}}_{k}\cdot\xi}: k=1, \ldots, N\}$,
 the above  linear independence  and  $\widehat{\varphi}(\textbf{0})>0$,   through the similar procedures in
\eqref{ouyang} and \eqref{zhengding1234}  one can prove  the integral in
\eqref{BVCNN} is positive.  Consequently, $\varphi$ is positive definite.

In what follows, we prove that $\nu_{2}(\varphi)\geq1$. First, we have
\begin{align}\label{by123}\begin{array}{lllllllll}
\displaystyle \int_{\mathbb{R}^{2}}|\widehat{\varphi}(\xi_{1}, \xi_{2})|^{2}(1+\xi^{2}_{1}+\xi^{2}_{2})d\xi_{1}d\xi_{2}& \displaystyle\leq\int_{|\xi_{1}|\leq1, \xi_{2}\in \mathbb{R}}|\widehat{\varphi}(\xi_{1}, \xi_{2})|^{2}(1+\xi^{2}_{1}+\xi^{2}_{2})d\xi_{1}d\xi_{2}\\
&\displaystyle+\int_{|\xi_{2}|\leq1, \xi_{1}\in \mathbb{R}}|\widehat{\varphi}(\xi_{1}, \xi_{2})|^{2}(1+\xi^{2}_{1}+\xi^{2}_{2})d\xi_{1}d\xi_{2}\\
&\displaystyle+\int_{|\xi_{1}|>1, \xi_{2}>1}|\widehat{\varphi}(\xi_{1}, \xi_{2})|^{2}(1+\xi^{2}_{1}+\xi^{2}_{2})d\xi_{1}d\xi_{2}\\
&:=I_{1}+I_{2}+I_{3}.
\end{array}
\end{align}
We first estimate $I_{1}$ as follows,
\begin{align}\begin{array}{lllllllll}
I_{1}&=\displaystyle\int_{|\xi_{1}|\leq1, \xi_{2}\in \mathbb{R}}|\widehat{\varphi}(\xi_{1}, \xi_{2})|^{2}(1+\xi^{2}_{1}+\xi^{2}_{2})d\xi_{1}d\xi_{2}\\
&=\displaystyle\int_{|\xi_{1}|\leq1, |\xi_{2}|\leq\sqrt{2}}|\widehat{\varphi}(\xi_{1}, \xi_{2})|^{2}(1+\xi^{2}_{1}+\xi^{2}_{2})d\xi_{1}d\xi_{2}\\
&+\displaystyle\int_{|\xi_{1}|\leq1, |\xi_{2}|>\sqrt{2}}|\widehat{\varphi}(\xi_{1}, \xi_{2})|^{2}(1+\xi^{2}_{1}+\xi^{2}_{2})d\xi_{1}d\xi_{2}\\
&:=I_{11}+I_{12}.
\end{array}
\end{align}
Recall  that  $\widehat{\varphi}\in C^{\infty}(\mathbb{R}^{2}).$
By the continuity of $\widehat{\varphi}$ we have $I_{11}<\infty$. Moreover,
\begin{align}\label{VCXBZ}\begin{array}{lllllllll}
I_{12}&=\displaystyle\int_{|\xi_{1}|\leq1, |\xi_{2}|>\sqrt{2}}|\widehat{\varphi}(\xi_{1}, \xi_{2})|^{2}(1+\xi^{2}_{1}+\xi^{2}_{2})d\xi_{1}d\xi_{2}\\
&\displaystyle\leq 2\int_{|\xi_{1}|\leq1}\Big[\frac{\sin\xi_{1}/2}{\xi_{1}/2}\Big]^{2n_{1}}d\xi_{1}
\int_{|\xi_{2}|>1}\Big[\frac{\sin\xi_{2}/2}{\xi_{2}/2}\Big]^{4n_{2}}\xi^{2}_{2}d\xi_{2} \ \ (\ref{VCXBZ} A)\\
&<\infty,
\end{array}
\end{align}
where (\ref{VCXBZ}$A$) is from  $1+\xi^{2}_{1}\leq \xi^{2}_{2}$.
Consequently, $I_{1}=I_{11}+I_{12}<\infty$. Similarly, one can prove that  $I_{2}<\infty$.
Additionally,
\begin{align}\begin{array}{lllllllll}
I_{3}&\leq 2\displaystyle\int_{|\xi_{1}|>1, \xi_{2}>1}|\widehat{\varphi}(\xi_{1}, \xi_{2})|^{2}(\xi^{2}_{1}+\xi^{2}_{2})d\xi_{1}d\xi_{2}\\
&\displaystyle\leq2\times 2^{2(n_{1}+n_{2})}\Big[\int_{|\xi_{1}|>1, \xi_{2}>1}\frac{1}{\xi^{4n_{1}-2}_{1}\xi^{4n_{1}}_{2}}d\xi_{1}d\xi_{2}+\int_{|\xi_{1}|>1, \xi_{2}>1}\frac{1}{\xi^{4n_{1}-2}_{2}\xi^{4n_{1}}_{1}}d\xi_{1}d\xi_{2}\Big]\\
&<\infty.
\end{array}
\end{align}
Now by \eqref{by123}, we have $\int_{\mathbb{R}^{2}}|\widehat{\varphi}(\xi_{1}, \xi_{2})|^{2}(1+\xi^{2}_{1}+\xi^{2}_{2})d\xi_{1}d\xi_{2}<\infty$,
and consequently, $\nu_{2}(\varphi)\geq1$. This completes the proof.
\end{proof}

\subsection{The fifth    main result:    pairs of  $(\varphi, \emph{\textbf{p}})$
such that $\{\emph{\textbf{p}}\emph{\textbf{k}}_{1}, \ldots, \emph{\textbf{p}}\emph{\textbf{k}}_{\#E}\}$
is eligible for SACT sampling, where
$\varphi$ is positive definite and vanishing.} \label{vanishingqingkuang}

It follows from Lemma \ref{characerization} that
for a continuous positive definite  function $\varphi$, its Fourier transform $\widehat{\varphi}$ is necessarily nonvanishing,
namely, $\varphi(\textbf{0})\neq0$. But $\varphi$ itself  is not necessarily nonvanishing, namely, $\widehat{\varphi}(\textbf{0})\neq0$
does not necessarily hold. That is, there exist positive definite and vanishing functions.
We next provide  an example to explain this. It is the motivation for this subsection.
\subsubsection{A motivation example}
\begin{exam}\label{lizi}
Let \begin{align}\phi_{1}(x_{1})=B_{2}(x_{1}+1)=\left\{\begin{array}{cccccccccc}
x_{1}+1,&-1<x_{1}\leq0,\\
1-x_{1}, &0<x_{1}<1,\\
0, &|x_{1}|\geq1.
\end{array}\right.\end{align}
By \eqref{fuliyebiaodashi} we have $\widehat{\phi_{1}}(\xi_{1})=(\frac{\sin\xi_{1}/2}{\xi_{1}/2})^{2}.$
 Define $\varphi_{1}$ via $$\widehat{\varphi_{1}}(\xi_{1})=(\frac{e^{\texttt{i}\xi_{1}/2}-e^{-\texttt{i}\xi_{1}/2}}{2\texttt{i}})^{2}\widehat{\phi}_{1}(\xi_{1}/2)=
\sin^{2}(\xi_{1}/2)(\frac{\sin(\xi_{1}/4)}{\xi_{1}/4})^{2}\geq0.$$
Additionally, in the time-domain  $\varphi_{1}(x_{1})=-\frac{1}{2}\phi_{1}(2x_{1}+1)+\phi_{1}(2x_{1})-\frac{1}{2}\phi_{1}(2x_{1}-1).$
It is straightforward to check that $\varphi_{1}$ is continuous and   bounded, and $\varphi_{1}(0)=1$. By Lemma \ref{characerization}, $\varphi_{1}$
is positive definite. But it is clear that $\widehat{\varphi_{1}}(0)=0.$
Now through the tensor product we define
 \begin{align}\label{op456}\varphi(x_{1},x_{2})=\varphi_{1}(x_{1})\varphi_{2} (x_{2}).\end{align}
Using Lemma \ref{characerization} again, one can check that $\varphi$ is also positive definite. But $\widehat{\varphi}(\textbf{0})=0.$
That is, $\varphi$ is vanishing.
\end{exam}

As summarized  in Remark \ref{xiangliangdezuoyong}
the nonvanishing property is key in Theorem \ref{ddgo} for  providing  great flexibility for the choice
of direction vector $\emph{\textbf{p}}.$
On the other hand, Example \ref{lizi} confirms  the existence of positive definite but vanishing functions, and such functions do not reach the requirement of
Theorem \ref{ddgo}.  As such, for the vanishing case we need to address what direction vector $\emph{\textbf{p}}$ is eligible for
the SACT sampling.

\subsubsection{The SACT sampling result when $\varphi$ is positive definite and vanishing}
Now it is ready to establish the fifth  main result in the following Theorem
\ref{eligiblevectors}.
On the generator, the difference between the  Theorem
\ref{eligiblevectors}  and Theorem \ref{ddgo}
is that the generator $\varphi$ here  is vanishing here, namely, $\widehat{\varphi}(\textbf{0})=0$
while that in Theorem \ref{ddgo} is nonvanishing. The following definition will be  necessary for Theorem
\ref{eligiblevectors}.

\begin{defi}\label{910876}
Let $\varphi: \mathbb{R}^2\rightarrow \mathbb{C}$ be  positive definite and compactly supported. For
$\emph{\textbf{x}}_{0}\in \mathbb{R}^{2}$ such that $\widehat{\varphi}(\emph{\textbf{x}}_{0})>0$,   as in Definition \ref{definitin1234}, $\delta^{\widehat{\varphi}}_{\emph{\textbf{x}}_{0}, \max}\in (0, \infty]$ is supposed to be  the maximum value such that $\widehat{\varphi}(\emph{\textbf{x}})>0$ for any $\emph{\textbf{x}}\in \mathring{\mathcal{D}}(\emph{\textbf{x}}_0,
\delta^{\widehat{\varphi}}_{\emph{\textbf{x}}_{0}, \max})$.
Denote the nonzero set of $\widehat{\varphi}$ by $\mathcal{G}_{\widehat{\varphi}}$ such that $\widehat{\varphi}(\emph{\textbf{x}})>0$
for any $\emph{\textbf{x}}\in \mathcal{G}_{\widehat{\varphi}}$. As in Definition \ref{flagvector},  define
\begin{align}\label{DVV}\hbox{DV}_{\widehat{\varphi}}=\bigcup_{\emph{\textbf{x}}\in \mathcal{G}_{\widehat{\varphi}}}\hbox{dv}_{\mathring{\mathcal{D}}(\emph{\textbf{x}},
\delta^{\widehat{\varphi}}_{\emph{\textbf{x}}, \max})},\end{align}
where $\hbox{dv}_{\mathring{\mathcal{D}}(\emph{\textbf{x}},
\delta^{\widehat{\varphi}}_{\emph{\textbf{x}}, \max})}$ is defined via  Definition \ref{definitin1234}.
\end{defi}

\begin{theo}\label{eligiblevectors}
Suppose that $\varphi: \mathbb{R}^2\longrightarrow \mathbb{C}$ is   compactly supported, continuous,  positive   and vanishing  such that $\hbox{supp}(\varphi)\subseteq[N_{1}, M_{1}]\times
[N_{2}, M_{2}]$, its Sobolev smoothness  $\nu_{2}(\varphi)>1/2$ and $\widehat{\varphi}(\textbf{0})=0$.
Moreover,  $f\in V(\varphi, \mathbb{Z}^2)$ is an arbitrary target function
such that $\hbox{supp}(f)\subseteq[a_{1}, b_{1}]\times [a_{2}, b_{2}]$. As previously,
define  $ E=\{\emph{\textbf{k}}_{1}, \ldots, \emph{\textbf{k}}_{\#E}\}=\big\{\big[\lceil a_{1}-M_{1}\rceil, \lfloor b_{1}-N_{1}\rfloor\big]\times
\big[\lceil a_{2}-M_{2}\rceil, \lfloor b_{2}-N_{2}\rfloor\big]\big\}\cap \mathbb{Z}^{2},$
and \begin{align}E^{+}=\left\{
\begin{array}{cccc}
   \emptyset,&\#E=1,\\
  \{\emph{\textbf{x}}-\emph{\textbf{y}}:
\emph{\textbf{x}}\neq\emph{\textbf{y}}\in E\},&\#E>1.
\end{array}
\right.\end{align}
Then $f$ can be determined uniquely by its SA Radon (w.r.t $\emph{\textbf{p}}$) samples
at $\{\emph{\textbf{p}}\emph{\textbf{k}}_{1}, \ldots, $
$\emph{\textbf{p}}\emph{\textbf{k}}_{\#E}\}$,
where $\emph{\textbf{p}}$ is an arbitrary     direction vector from $\hbox{DV}_{\widehat{\varphi}}\setminus\hbox{dv}_{\mathcal{N}_{E^{+}}}$,
with $\hbox{dv}_{\mathcal{N}_{E^{+}}}$ defined in Definition \ref{crictial}.
\end{theo}
\begin{proof}
It has been proved in the proof of Theorem \ref{yibanshengcyuanYY} that $\hbox{DV}_{\widehat{\varphi}}\setminus\hbox{dv}_{\mathcal{N}_{E^{+}}}$
is not empty.
Since the only difference between the generator $\varphi$ here  and that in Theorem \ref{ddgo} is
the vanishing property
$\widehat{\varphi}(\textbf{0})=0,$ we simplify the proof and focus on  something related to
the  difference. Firstly, since  $\nu_{2}(\varphi)>1/2$ then  it follows from
 Proposition \ref{lianxuxing} (2) that $\mathcal{R}_{\emph{\textbf{p}}}\varphi$ is continuous for any $\emph{\textbf{p}}$.
Secondly, for the case that $\#E>1$ as in the proof of Theorem \ref{ddgo} one can check that for
any direction vector $\emph{\textbf{p}}\in \{(\cos\theta, \sin\theta): \theta\in [0, 2\pi)\}\setminus\hbox{dv}_{\mathcal{N}_{E^{+}}}$,
we have $\emph{\textbf{p}}\emph{\textbf{k}}_{l}\neq
\emph{\textbf{p}}\emph{\textbf{k}}_{n}$ for any $l\neq n\in \{1, \ldots, \#E\}$.
Then item (iii) of Theorem \ref{HHKK} holds.
 Now we focus on the proof that  $\widehat{\mathcal{R}_{\emph{\textbf{p}}}\varphi}\geq0$ and
 $\widehat{\mathcal{R}_{\emph{\textbf{p}}}\varphi}\not\equiv0$ for any
 $\emph{\textbf{p}}\in \hbox{DV}_{\widehat{\varphi}}.$ By $\widehat{\varphi}\geq0$ we have $\widehat{\mathcal{R}_{\emph{\textbf{p}}}\varphi}(\xi)=
 \widehat{\varphi}(\emph{\textbf{p}}^{T}\xi)\geq0$ for any
 $\emph{\textbf{p}}\in \hbox{DV}_{\widehat{\varphi}}.$ Since $\emph{\textbf{p}}\in \hbox{DV}_{\widehat{\varphi}},$
  it follows from Proposition  \ref{lebesguecedu} (2) that
$\mathcal{R}_{\emph{\textbf{p}}}\varphi\not\equiv0.$ Then item (i)
of Theorem \ref{HHKK} holds.
Now by Theorem \ref{HHKK}, the proof is completed.
\end{proof}

\end{document}